\documentclass[11pt]{amsart}

\usepackage{epigamath}
\usepackage{microtype}

\usepackage[english]{babel}

\numberwithin{equation}{section}

\usepackage{amscd,graphicx,epsfig}

\usepackage[normalem]{ulem}
\let\isout\sout
\renewcommand{\sout}[1]{\ifmmode\text{\isout{\ensuremath{#1}}}\else\isout{#1}\fi}

\newtheorem{thm}{Theorem}[section]

\newtheorem{prop}[thm]{Proposition}
\newtheorem{lem}[thm]{Lemma}
\newtheorem{cor}[thm]{Corollary}

\theoremstyle{definition}
\newtheorem{defn}[thm]{Definition}
\newtheorem{exa}[thm]{Example}
\newtheorem*{thanks*}{Acknowledgments}

\newtheorem*{rem*}{Remark}
\newtheorem{rem}[thm]{Remark}

\newcommand{\DDD}{\mathcal D}
\newcommand{\FFF}{\mathcal F}
\newcommand{\NNN}{\mathcal N}
\newcommand{\VVV}{\mathcal V}
\newcommand{\WWW}{\mathcal W}
\newcommand{\E}{\mathcal E}
\newcommand{\LLL}{\mathcal L}
\newcommand{\GGG}{\mathcal G}
\newcommand{\SSS}{\mathcal S}
\newcommand{\eps}{\varepsilon}
\renewcommand{\phi}{\varphi}
\newcommand{\VFG}{\VF_{G}}
\newcommand{\di}{\frac{\partial}{\partial x_{i}}}
\newcommand{\bk}{\Bbbk}
\newcommand{\kst}{\Bbbk^{*}}
\newcommand{\NN}{\mathbb N}
\newcommand{\PP}{\mathbb P}
\newcommand{\RR}{\mathbb R}
\newcommand{\CC}{\mathbb C}
\renewcommand{\AA}{\mathbb A}
\newcommand{\mm}{\mathfrak m}
\newcommand{\pp}{\mathfrak p}
\newcommand{\Aone}{\AA^{1}}
\renewcommand{\subset}{\subseteq}
\newcommand{\onto}{\twoheadrightarrow}
\newcommand{\into}{\hookrightarrow}
\newcommand{\SLtwo}{\SL_{2}}
\newcommand{\quot}{/\!\!/}
\newcommand{\nn}{\left[\begin{smallmatrix} 0 & 1 \\ 0 & 0\end{smallmatrix}\right]}
\newcommand{\name}[1]{\textsc{#1\/}}
\newcommand{\simto}{\xrightarrow{\sim}}
\newcommand{\be}{\begin{enumerate}}
\newcommand{\ee}{\end{enumerate}}

\renewcommand{\gg}{\mathfrak g}
\newcommand{\uu}{\mathfrak u}
\newcommand{\Om}{O_{\text{\it min}}}
\newcommand{\bOm}{\overline{\Om}}
\newcommand{\ps}{\par\smallskip}
\newcommand{\ms}{\par\medskip}

\newcommand{\g}[3]{\left[\begin{smallmatrix}1&#1&#2\\&1&#3\\&&1\end{smallmatrix}\right]}
\newcommand{\vv}[3]{\left[\begin{smallmatrix}0&#1&#2\\&0&#3\\&&0\end{smallmatrix}\right]}

\newcommand{\VG}{\VF_G}
\newcommand{\VSL}{\VF_{\SLtwo}}

\DeclareMathOperator{\VF}{Vec} \DeclareMathOperator{\Der}{Der}
\newcommand{\Derc}{\Der^{c}}
\DeclareMathOperator{\End}{End} \DeclareMathOperator{\Aut}{Aut}
\DeclareMathOperator{\Mor}{Mor} \DeclareMathOperator{\id}{id}
\DeclareMathOperator{\Ad}{Ad} \DeclareMathOperator{\M}{M}
\DeclareMathOperator{\codim}{codim} \DeclareMathOperator{\Gr}{Gr}
\DeclareMathOperator{\Lie}{Lie} \DeclareMathOperator{\tdeg}{tdeg}
\DeclareMathOperator{\ad}{ad} \DeclareMathOperator{\Spec}{Spec}
\DeclareMathOperator{\Cov}{Cov}
\DeclareMathOperator{\sltwo}{{\mathfrak sl}_{2}}
\DeclareMathOperator{\SL}{SL} \DeclareMathOperator{\GL}{GL}
\newcommand{\GLn}{\GL_{n}}
\newcommand{\Ftt}{{}_3F_2}
\DeclareMathOperator{\OOO}{\mathcal O}
\DeclareMathOperator{\Norm}{Norm} \DeclareMathOperator{\res}{res}

\EpigaVolumeYear{4}{2020} \EpigaArticleNr{13} \ReceivedOn{December
16, 2019}
\InFinalFormOn{June 17, 2020} \AcceptedOn{July 25, 2020}

\title{Covariants, derivation-invariant subsets, and first integrals}
\titlemark{Covariants, derivation-invariant subsets, and first integrals}

\author{Frank Grosshans}
\address{Department of Mathematics, West Chester University, West Chester, PA 19383, USA}
\email{fgrosshans@wcupa.edu}
\author{Hanspeter Kraft}
\address{Departement Mathematik und Informatik,
Universit\"at Basel, Spiegelgasse 1, CH-4051 Basel, Switzerland}
\email{Hanspeter.Kraft@unibas.ch}

\authormark{F. Grosshans and H. Kraft}

\AbstractInEnglish{\scriptsize{Let $\bk$ be an algebraically closed field of
characteristic $0$, and let $V$ be a finite-dimensional vector
space. Let $\End(V)$ be the semigroup of all polynomial
endomorphisms of $V$. Let $\E\subset \End(V)$ be a linear subspace
which is also a subsemigroup. Both $\End(V)$ and $\E$ are
ind-varieties which act on $V$ in the obvious way.

In this paper, we study important aspects of such actions. We assign
to $\E$ a linear subspace $\DDD_{\E}$ of the vector fields on $V$. A
subvariety $X$ of $V$ is said to be $\DDD_{\E}$-invariant if
$\xi(x)\in T_{x}X$ for all $\xi \in\DDD_{\E}$. We show that $X$ is
$\DDD_{\E}$ -invariant if and only if it is the union of
$\E$-orbits. For such $X$, we define first integrals and show that
they are the rational functions on a certain ``quotient'' of $X$
defined by the action of $\E$.

An important case occurs when $G$ is an algebraic subgroup of
$\GL(V)$ and $\E$ consists of the $G$-equivariant polynomial
endomorphisms. In this case, the associated $\DDD_{\E}$ is the space
of $G$-invariant vector fields. A significant question here is
whether there are non-constant $G$-invariant first integrals on $X$.
As examples, we study the adjoint representation, \name{Luna}
strata, the orbit closures of highest weight vectors, and
representations of the additive group. We also look at
finite-dimensional irreducible representations of $\SL_{2}$ and
their nullcones.}}



\TitleInFrench{Covariants, sous-ensembles invariants par
d\'erivation et int\'egrales premi\`eres}

\AbstractInFrench{\scriptsize{Soient $\bk$ un corps alg\'ebriquement clos de
caract\'eristique nulle et $V$ un espace vectoriel de dimension
finie. Soit $\End(V)$ le semi-groupe des endomorphismes polynomiaux
de $V$. Soit $\E\subset \End(V)$ un sous-espace lin\'eaire qui est
aussi un semi-groupe. Ainsi $\End(V)$ et $\E$ sont des
ind-vari\'et\'es qui agissent naturellement sur $V$.

Dans cet article, nous \'etudions des aspects importants de ces
actions. Nous associons \`a $\E$ un sous-espace lin\'eaire
$\DDD_{\E}$ form\'es de champs de vecteurs sur $V$. Une sous-vari\'et\'e $X$
de $V$ est dite $\DDD_{\E}$-invariante si $\xi(x)\in T_{x}X$ pour
tout $\xi \in\DDD_{\E}$. Nous montrons que $X$ est
$\DDD_{\E}$-invariante si et seulement si elle est r\'eunion de
$\E$-orbites. Pour une telle sous-vari\'et\'e $X$, nous
d\'efinissons des int\'egrales premi\`eres et montrons que ce sont
des fonctions rationnelles sur un certain ``quotient'' de $X$
d\'efini par l'action de $\E$.

Un cas particulier se pr\'esente lorsque $G$ est un sous-groupe alg\'ebrique de
$\GL(V)$ et $\E$ est form\'e par les endomorphismes polynomiaux
$G$-\'equivariants. Dans ce cas, $\DDD_{\E}$ est l'espace des champs
de vecteurs $G$-invariants. Ici, une question naturelle porte sur
l'existence d'int\'egrales premi\`eres $G$-invariantes non
constantes sur $X$. Comme exemples, nous \'etudions la
repr\'esentation adjointe, les strates de \name{Luna}, les
adh\'erences d'orbites de vecteurs de plus haut poids et les
repr\'esentations du groupe additif. Nous consid\'erons aussi les
repr\'esentations irr\'eductibles de dimension finie de $\SL_{2}$ et
leurs nilc\^ones.}}

\begin{document}


\removeabove{0.5cm}
\removebetween{0.5cm}
\removebelow{0.5cm}

\maketitle

\begin{prelims}

\DisplayAbstractInEnglish

\medskip

\DisplayKeyWords

\medskip

\DisplayMSCclass

\bigskip

\languagesection{Fran\c{c}ais}

\medskip

\DisplayTitleInFrench

\DisplayAbstractInFrench

\end{prelims}



\setcounter{tocdepth}{1}

\tableofcontents


\section{Introduction and main results}
This paper is concerned with the relationship between three
concepts: derivation-invariant subsets, endomorphisms of an affine
variety $X$, and first integrals. We show that this relationship has
features similar to those of algebraic group actions with first
integrals playing the role of invariant functions. Let $\bk$ be an
algebraically closed field of characteristic $0$ and let $X$ be an
irreducible affine variety over $\bk$. Let $\DDD\subseteq \VF(X)$ be
a set of algebraic vector fields on $X$.  A closed subvariety
$Y\subseteq X$ is called $\DDD$-invariant if $\xi(y)\in T_{y}Y$ for
all $y\in Y$ and $\xi\in\DDD$, i.e., $\xi$ is tangent to $Y$ at
every point of $Y$. We establish some basic properties of invariant
subsets including the following: {\it For any $x\in X$, there is a
smallest $\DDD$-invariant closed subvariety, $M(x)$, which contains
$x$\/} (Lemma~\ref{D-invariant.lem}). A first integral of $\DDD$ is
a rational function $f\in\bk(X)$ such that $\xi f=0$ for all
$\xi\in\DDD$. We show that first integrals are precisely those
functions which are constant on the closed subsets $M(x)$
(Lemma~\ref{first-integral.lem}).

Our new idea is to consider the semigroup $\End(X)$ consisting of
all endomorphisms of the variety $X$ and to use the important fact
that $\End(X)$ is a so-called {\it ind-variety}. This allows us to
define the (Zariski) tangent space $T_{\id}\End(X)$ and to
associate to any $A\in T_{\id}\End(X)$ a vector field $\xi_{A}$ on
$X$ in the usual way, see Section~\ref{endo.subsec}. For a closed
subsemigroup $\E\subset \End(X)$ we denote by $\DDD_{\E}\subset
\VF(X)$ the set of associated vector fields. There is a natural
action of $\End(X)$ on $X$, $(\phi,x)\mapsto \phi(x)$, and the  {\it
$\E$-orbit\/} of an element $x\in X$ is defined as
$\E(x):=\{\phi(x)\mid \phi\in \E\}$. We first show that if a closed
subvariety $Y\subset X$ is the union of $\E$-orbits, then $Y$ is
$\DDD_{\E}$-invariant (Proposition~\ref{E&D-invariance.prop}).

Now suppose that $V$ is a finite-dimensional vector space and that
$X\subseteq V$. Suppose also that $\E \subseteq \End(X)$ is a {\it
linear subspace\/} which is also a subsemigroup. We show that for
$x\in X$, $\E(x)=M(x)$ and that a subvariety $Y\subseteq X$ is
$\DDD_{\E}$ -invariant if and only if it is a union of $\E$-orbits
(Theorem~\ref{main-theorem-1}). This means that the
$\DDD_{\E}$-invariant  subvarieties are precisely those which are
stable under the action of $\E$.

Furthermore, there is an open dense subset $X' \subseteq X$ so that
first integrals separate the various $\E(x)\cap X'$ and thus can be
regarded as the rational functions on a certain ``quotient space''
$X\quot \E$ for the action of $\E$ on $X$
(Proposition~\ref{quotient-mod-E.prop}). This construction includes
an algebraic (and global) version of a classical theorem of
\name{Frobenius}  \cite[Theorem 1.60]{Wa1971Foundations-of-dif}.

The most important example of the above occurs when $V$ is a
finite-dimensional vector space, $G \subseteq \GL(V)$ an algebraic
group,
$$
\End_{G}(V) := \{\phi\in \End(V)\mid \phi(g\cdot v) = g \cdot
\phi(v)\text{ for all } g\in G \text{ and } v\in V\},
$$
the semigroup of {\it covariants}, $X \subseteq V$ an
$\End_G(V)$-stable closed subvariety and $\E:=\End_G(V)|_X$.  When
$X$ is $G$-stable, an important question is whether or not there are
non-constant $G$-invariant first integrals on $X$. Examples show
that such can occur. However, in those cases where they do not, the
field of first integrals is the field of rational functions on a
homogeneous space (Lemma~\ref{first-integrals.lem}). Furthermore,
the $\E$-$G$-orbit of a generic point is open and dense in $X$.

When $G$ is reductive and the orbit $Gx$ is closed, \name{Panyushev}
has shown that  $\E(x)=X^{G_x}$ (see
Proposition~\ref{Panyushev.prop}). Thus, first integrals separate
open subsets of the various $X^{G_x}$. Furthermore, when the generic
$G$-orbit in $X$ is closed, we show that there are no $G$-invariant
first integrals (Theorem~\ref{first-integrals.thm}).

Finally, in Section~\ref{SLtwo.sec}, we study the case where
$G=\SLtwo$ and either $X=V_{d}$,  the binary forms of degree $d$, or
$X$ is the nullcone $\NNN_{d}$ of $V_{d}$. In the latter case, there
are no closed orbits other than $\{0\}$ and the main problem in
finding the $\E(v)$ is the construction of covariants.


Algebraically, derivation-invariant ideals have long been of
interest \cite{Se1967Differential-ideal}. In the context of ordinary
differential equations, it is well known that a Zariski-closed set
$X\subset V$ is $\DDD_\E$-invariant if and only if it is the union
of trajectories of solutions to $\frac{dx}{dt}=\xi(x)$, $\xi\in\E$
(see Lemma 2.1 below). The study of the subsets $\E(v)$ began with
the paper \cite{LeSp1999A-note-concerning-} by
\name{Lehrer-Springer} which was subsequently extended by
\name{Panyushev} \cite{Pa2002On-covariants-of-r}. These papers,
however, did not draw the connection to derivation-invariant
subsets. For vector spaces, that connection appears in \cite[Theorem
3.6]{GrScWa2012Invariant-sets-for}. The difficult problem of
constructing the module of covariants on $X$ was first considered in
the classical invariant theory of the nineteenth century
\cite{El1913An-Introduction-to} and continues to be of active
research interest. It is worth noting that when $G$ is reductive and
$X\subset V$ is an $\End_{G}(V)$- and $G$-stable variety, then
$\max\{\dim\E(v)\mid v\in X\}$, which is calculated in many of our
examples, is (easily) shown to be the rank of the module of
covariants from $X$ to $X$.

\ms
\begin{thanks*}
The first author thanks Sebastian Walcher for his advice on an earlier version of this paper.
\end{thanks*}

\section{Basic material}
\subsection{Vector fields and $\DDD$-invariant subsets}
Our base field $\bk$ is algebraically closed of characteristic zero.
We start with a lemma which translates the concept of invariant
subsets with respect to an ordinary differential equation into the
algebraic setting. For an affine variety $X$ an {\it algebraic
vector field} $\xi=(\xi(x))_{x\in X}$ is a collection of tangent
vectors $\xi(x)\in T_{x}X$ such that, for every regular function $f
\in \OOO(X)$, the function $\xi f\colon x \mapsto \xi(x)f$ is again
regular. It is easy to see that this is the same as a derivation of
the coordinate ring $\OOO(X)$. Note that $\xi f$ is also defined for
rational functions $f$.

In addition, one can define the  the {\it tangent bundle\/} $TX$ of
$X$ which is a variety together with a projection $p\colon TX \to X$
such that the fibers $p^{-1}(x)$ are the Zariski tangent spaces
$T_{x}X$. Then the sections are the algebraic vector fields (see
e.g. \cite[Appendix A.4.5]{Kr2016Algebraic-Transfor}). It is clear
that the algebraic vector fields form a $\OOO(X)$-module which will
be denoted by $\VF(X)$ and which can be identified with the
$\OOO(X)$-module $\Der(\OOO(X))$ of derivations of $\OOO(X)$. The
next lemma seems to be well known; a somewhat different proof may be
found in \cite[Lemma~A.1]{SaScWa2015Motion-in-a-symmet}.

\begin{lem}
Let $X$ be a smooth complex variety, and let $\xi \in \VF(X)$ be an
algebraic vector field. Then a Zariski-closed subvariety $Y \subset
X$ is invariant with respect to the flow defined by the differential
equation $\dot x = \xi(x)$ if and only if $\xi(y) \in T_{y}Y$ for
all $y \in Y$.
\end{lem}
\begin{proof}
Let $\Phi\colon X \times \RR \to X$ be the local flow of $\xi$,
defined in an open neighborhood of $X \times \{0\}$. By definition,
$$
\frac{\partial }{\partial t} \Phi(x,t)|_{t=0} = \xi(x) \text{ for
all } x \in X.
$$
This implies that if $Y$ is invariant under $\Phi$, then $\xi(y) \in
T_{y}Y$ for all $y \in Y$. On the other hand, assume that $\xi(y)
\in T_{y}Y$ for all $y \in Y$, and denote by $Y' \subset Y$ the open
dense set of smooth points of $Y$. Then $\xi|_{Y'}$ defines a local
flow $\Phi_{Y'}\colon Y' \times \RR \to Y'$ such that
$\frac{\partial }{\partial t} \Phi_{Y'}(y',t)|_{t=0} = \xi(y')$ for
all  $y \in Y'$. By the uniqueness of the local flow, we have
$\Phi_{Y'} = \Phi|_{Y' \times \RR}$, and so $Y'$ is invariant under
$\Phi$. Since $Y = \overline{Y'}$ we see that $Y$ is also invariant
under $\Phi$.
\end{proof}
This lemma allows to define the invariance of subvarieties with
respect to a set of vector fields for an arbitrary $\bk$-variety
$X$.
\begin{defn}
Let $\DDD \subset \VF(X)$ be a set of vector fields. \be
\item
A closed subvariety $Y \subset X$ is called {\it $\DDD$-invariant\/}
if $\xi(y)\in T_{y}Y$ for all $y\in Y$ and all $\xi \in \DDD$. We
also say that the vector fields $\xi \in \DDD$ are {\it parallel to
$Y$}.
\item
A subspace $W \subset \OOO(X)$ is called {\it $\DDD$-invariant} if
$\xi(W) \subset W$ for all $\xi \in \DDD$. \ee
\end{defn}
\begin{rem}
We will constantly use the following easy fact. If $\xi$ is a vector
field parallel to $Y$ and $f$  a rational function on $X$ defined in
a neighborhood $U$ of $y \in Y$, then $\xi(y) f = \xi(y) (f|_{U\cap
Y})$. In particular, if $f$ is regular on $X$, then $(\xi
f)|_{Y}=\xi|_{Y} (f|_{Y})$.
\end{rem}

\ps
\subsection{$\DDD$-invariant ideals}
Let $\DDD \subset \VF(X)$ be a set of vector fields.
\begin{lem}
If $I(Y) \subset \OOO(X)$ denotes the vanishing ideal of $Y$, then
$Y$ is $\DDD$-invariant if and only if $I(Y)$ is $\DDD$-invariant.
\end{lem}
\begin{proof}
If $f \in I(Y)$, then, for $y \in Y$,  $(\xi f)(y) = \xi(y) f =
\xi(y) f|_{Y}= 0$, hence $\xi f \in I(Y)$. Conversely, if $\xi(I(Y))
\subset I(Y)$, then $\xi$ induces a derivation of $\OOO(X)/I(Y) =
\OOO(Y)$, and the claim follows.
\end{proof}

For a closed subvariety $Y \subset X$ we can define the  Lie
subalgebra of the vector fields on $X$ parallel to $Y$:
$$
\VF_{Y}(X):=\{\xi \in \VF(X) \mid \xi(y) \in T_{y}Y \text{ for all
}y \in Y\} \subset \VF(X).
$$
We have a homomorphism of Lie algebras
$$
\rho\colon \VF_{Y}(X) \to \VF(Y), \quad\xi\mapsto \xi|_{Y},
$$
whose kernel consists of the vector fields on $X$ vanishing on $Y$.
The homomorphism $\rho$ is surjective when $X$ is a vector space.
With this notation we see that  $Y$ is $\DDD$-invariant if and only
if  $\DDD \subset \VF_{Y}(X)$. Part (3) of the next lemma can be
found in \cite[Theorem~1]{Se1967Differential-ideal}.

\begin{lem}{\label{D-invariant.lem}}
\leavevmode
\begin{enumerate}
\item Sums and intersections of $\DDD$-invariant ideals are $\DDD$-invariant.
\item If $I \subset \OOO(X)$ is a $\DDD$-invariant ideal, then so is $\sqrt{I}$.
\item If $Y_{i}\subset X$, $i\in I$, are $\DDD$-invariant closed subvarieties, then so is $\bigcap_{i\in I}Y_{i}$.
\item For any $x \in X$ there is a uniquely defined minimal $\DDD$-invariant closed subvariety $M(x)\subseteq X$ containing $x$.
\item If the closed subvariety $Y \subset X$ is $\DDD$-invariant, then every irreducible component of $Y$ is $\DDD$-invariant.
\end{enumerate}
\end{lem}
\begin{proof}
(1) is clear, (3) follows from (1) and (2), and (4) follows from
(3). \ps 
(2) It suffices to show that if $f^{n}=0$, then $(\xi
f)^{m}=0$ for some $m>0$. Let $e_{0}\geq 0$ be the minimal $e$ such
that there exists a $q\geq 0$ with $f^{e}\cdot(\xi f)^{q}=0$. If
$e_{0}=0$, we are done. So assume that $e_{0}>0$. Then
$$
0=\xi(f^{e_{0}}\cdot(\xi f)^{q})\cdot\xi f =
e_{0}f^{e_{0}-1}\cdot(\xi f)^{q+1} + q f^{e_{0}}\cdot(\xi
f)^{q}\cdot \xi^{2}f = e_{0}f^{e_{0}-1}\cdot(\xi f)^{q+1},
$$
contradicting the minimality of $e_{0}$. \ps 
(5) It suffices to
consider the case where $Y=X$, hence $(0) = \pp_{1}\cap\ldots\cap
\pp_{k}$ where the $\pp_{i}$ are the minimal primes of $\OOO(X)$.
For every $i$ choose an element $p_{i}\in\bigcap_{j\neq
i}\pp_{j}\setminus\pp_{i}$. Then $\pp_{i}=\{p\in\OOO(X)\mid p_{i} p
= 0\}$, and the same holds for every power of $p_{i}$. For every $p
\in \pp_{i}$ we find
$$
0 = p_{i} \xi(p_{i} p) = p_{i}(p_{i} \xi p + p\xi p_{i}) = p_{i}^{2}
\xi p,
$$
hence $\xi p \in \pp_{i}$.
\end{proof}

\begin{defn}
The closed subvarieties $M(x) \subset X$ from
Lemma~\ref{D-invariant.lem}(4) are called {\it minimal
$\DDD$-invariant subvarieties}. By Lemma~\ref{D-invariant.lem}(5)
they are irreducible.
\end{defn}

\ps
\subsection{Linear spaces of vector fields}
In the following, we will mainly deal with the case where $\DDD
\subset \VF(X)$ is a linear subspace. In this case, we set
$$
\DDD(x):= \eps_{x}(\DDD):=\{\xi(x) \mid \xi \in \DDD\} \subset
T_{x}X
$$
where  $\eps_{x}\colon \VF(X) \to T_{x}X$ is the (linear) evaluation
map $\xi\mapsto \xi(x)$. Note that a closed subvariety $Y \subset X$
is $\DDD$-invariant if and only if $\DDD(y) \subset T_{y}Y$ for all
$y \in Y$.

The following lemma is clear.
\begin{lem}\label{lower-semi-cont.lem}
For a linear subspace $\DDD \subset \VF(X)$ the function $x \mapsto
\dim\DDD(x)$ is lower semicontinuous, i.e., for every $x\in X$ the
set
$$
U_{x}:=\{u \in X\mid \dim\DDD(u)\geq \dim\DDD(x)\}
$$
is a (Zariski-) open neighborhood of $x$.
\end{lem}
Setting $d_{\DDD}(X):=\max_{x\in X}\dim \DDD(x)$ the lemma implies
that
$$
X':=\{x\in X \mid \dim\DDD(x)=d_{\DDD}(X)\}
$$
is open (and non-empty) in $X$.

\ms
\section{Endomorphisms}
\subsection{The semigroup of endomorphisms}\label{endo.subsec}
We now study the semigroup $\End(X)$ of endomorphisms  of $X$. An
important fact is that $\End(X)$ is an {\it ind-variety} (see
Appendix) which allows to define the (Zariski) tangent space
$T_{\id}\End(X)$. We have a canonical inclusion
$$
\Xi\colon T_{\id}\End(X) \into \VF(X), \  A \mapsto \xi_{A},
$$
where the vector field $\xi_{A}$ is defined in the following way
(see Appendix, Proposition~\ref{vector-fields.prop}). For any $x \in
X$ consider the ``orbit map'' $\mu_{x}\colon \End(X) \to X$,
$\phi\mapsto \phi(x)$, and its differential
$$
(d\mu_{x})_{\id}\colon T_{\id}\End(X) \to T_{x}X.
$$
Then define $\xi_{A}(x):=(d\mu_{x})_{\id}(A)$.

The canonical ``evaluation map'' $(\phi,x) \mapsto \phi(x)$ defines
a morphism of ind-varieties
$$
\Phi\colon \End(X) \times X \to X, \quad(\phi,x)\mapsto \phi(x)
$$
with the usual properties:
$$
\Phi(\id,x) = x \text{ \ and \ }  \Phi(\phi\circ\psi,x) =
\Phi(\phi,\Phi(\psi,x))
$$
for  all $x\in X$ and $ \phi,\psi \in \End(X)$. We will call this an
{\it action of the semigroup $\End(X)$ on $X$,} although there are
some major differences to group actions as we will see below.

For the differential of $\Phi$ we find
\[\tag{$*$}
d\Phi_{(\id,x_{0})}\colon T_{\id}\End(X) \oplus T_{x_{0}}X \to
T_{x_{0}}X, \quad (A,\delta)\mapsto \xi_{A}(x_{0})+\delta.
\]
\begin{defn}
If $\E \subset \End(X)$ is a closed subsemigroup we say that a
subset $Y \subset X$ is {\it stable under $\E$} (shortly
$\E$-stable), if $\phi(y) \in Y$ for all $y\in Y$ and all
$\phi\in\E$.

\begin{rem}\label{orbit.rem}
The stability under $\E$ can be expressed differently by using the
subsets
$$
\E(y):=\{\phi(y) \mid \phi\in \E\}
$$
which will be called {\it orbits of $y$ under $\E$}. Namely, {\it
$X$ is stable under $\E$ if and only if $X$ contains with every $y$
the orbit $\E(y)$.} However, one has to be very careful in using
this analogy with group actions since orbits under $\E$ are not
necessarily disjoint and they do not define a partition of $X$.
\end{rem}

\end{defn}
The closed subsemigroup $\E \subset \End(X)$ defines a linear
subspace $\DDD_{\E}\subset\VF(X)$ as the image of the tangent space
$T_{\id}\E$ under $\Xi$:
$$
\DDD_{\E}:= \{\xi_{A}\mid A \in T_{\id}\E\} \subset \VF(X).
$$
The main point of this section is to relate the invariance under
$\DDD_{\E}$ with the stability under the semigroup $\E$. A first and
easy result is the following.

\begin{prop}\label{E&D-invariance.prop}
Let $\E \subseteq \End(X)$ be a closed subsemigroup. If $Y \subset
X$ is a closed $\E$-stable subvariety, then $Y$ is
$\DDD_{\E}$-invariant.
\end{prop}
\begin{proof}
Since $Y$ is $\E$-invariant we have a morphism $\Phi\colon \E \times
Y \to Y$ whose differential
$$
d\Phi_{(\id,y)}\colon T_{\id} \E \oplus T_{y}Y \to T_{y}Y
$$
sends $(A,0)$ to $\xi_{A}(y)$, by  formula $(*)$ above. Thus $\xi(y)
\in T_{y}Y$ for all $\xi \in \DDD_{\E}$ which means that $Y$ is
$\DDD_{\E}$-invariant.
\end{proof}
We will see below that under stronger assumptions on $\E$ the
reverse implication also holds, i.e., a closed subset $Y \subset X$
is $\E$-stable if and only if it is $\DDD_{\E}$-invariant.
\begin{rem}
We do not know what the structure of the subsets $\E(x) \subset X$
is. If $\E$ is curve-connected (i.e. any two points of $\E$ can be
connected by an irreducible curve, see
Definition~\ref{affine-algebraic.def}(5)), then one can show that
$\E(x)$ contains a set $U$ which is open and dense in
$\overline{\E(x)}$. But it is not clear whether $\E(x)$ is
constructible.
\end{rem}

\ps
\subsection{The case of a vector space}
In case of a vector space $X = V$ the situation is much simpler,
because we can identify every tangent space $T_{v}V$ with $V$. In
particular,  vector fields $\xi\in\VF(V)$ correspond to morphisms
$\xi\colon V\to V$. Choosing a basis of $V$ we have
$$
\xi = \sum_{i=1}^{n}p_{i}\di \text{ where }p_{i} := \xi x_{i}.
$$
In this situation, the semigroup $\End(V) = \OOO(V) \otimes V$ is a
vector space, hence $T_{\id}\End(V) = \End(V)$ in a canonical way,
and
$$
\Xi\colon \End(V) = T_{\id}\End(V) \simto \VF(V)
$$
is the obvious isomorphism given as follows. In terms of coordinates
an endomorphism $\phi$ has the form $\phi=(p_{1},\ldots,p_{n})\colon
\bk^{n} \to \bk^{n}$ where $p_{i}=\phi^{*}(x_{i})$, and the
corresponding vector field $\xi:=\Xi(\phi)$ is given by  $\xi=
\sum_{i=1}^{n}p_{i}\di$.

The same formula holds for a semigroup  $\E \subset \End(V)$ which
is a  {\it linear} subspace. However, for a general closed semigroup
$\E\subset\End(V)$, we cannot identify $\E$ with $T_{\id}\E$, and so
the formula above does not make sense. For example, if
$\phi\in\End(V)$ is any endomorphism, then the semigroup
$\E:=\{\id,\phi,\phi^{2},\ldots\}$ is discrete, hence $T_{\id}\E$ is
trivial, and so $\DDD_{\E}$ is also trivial.
\par\smallskip
The following result is crucial. We will identify $T_{v}V$ with $V$
and thus consider the subspace $\DDD(v) \in T_{v}V$ as a subspace of
$V$.

\begin{lem}\label{main-lemma}
Let $\E \subset \End(V)$ be a linear subspace which is a semigroup.
Then
$$
\E(v) = M(v) = \DDD_{\E}(v) \ \text{ for all } v \in V.
$$
In particular, a subset $Y \subset V$ is $\DDD_{\E}$-invariant if
and only if it is $\E$-stable. Moreover, we have $\E(w) = \E(v)$ for
all $w$ in an open neighborhood of $v$ in $\E(v)$.
\end{lem}

\begin{proof}
(a) We have seen in Proposition~\ref{E&D-invariance.prop} that
$\E(v)$ is $\DDD_{\E}$-invariant, because it is stable under $\E$.
Hence, $\DDD_{\E}(w) \subset T_{w}\E(v)$ for all $w \in \E(v)$.
\par\smallskip
(b) The evaluation map $\mu_{v}\colon \E \to V$ is linear with image
$\E(v)$, hence $\E(v) \subset V$ is a linear subspace and
$\DDD_{\E}(v) = T_{v}\E(v) = \E(v)$.
\par\smallskip
(c) By Lemma~\ref{lower-semi-cont.lem} there is an open neighborhood
$U_{v}$ of $v$ in $\E(v)$ such that $\dim\DDD_{\E}(w) \geq
\dim\DDD_{\E}(v)$ for all $w \in U_{v}$. Hence, $\E(v)=\DDD_{\E}(v)
= \DDD_{\E}(w)=\E(w)$ for $w \in U_{v}$, by (a).
\par\smallskip
(d) It remains to prove the minimality, i.e. that $\E(v) = M(v)$.
Let $Y \subset \E(v)$ be closed and $\DDD_{\E}$-invariant with $v
\in Y$. Then, for every $w \in U_{v}\cap Y$, we have $\E(v) =
\DDD_{\E}(w) \subset T_{w}Y \subset \E(v)$. Hence, $\dim Y \geq \dim
\E(v)$, and so $Y = \E(v)$.
\end{proof}

\begin{rem}\label{closed-cone.rem}
If $\E \subset \End(V)$ is as in the lemma above, then it contains
the scalar multiplication $\bk \cdot\id$, and so $\E(v) \supset \bk
v$ for all $v \in V$. Therefore, every $\DDD_{\E}$-invariant closed
subvariety $X$ is a closed cone, i.e., contains with every point
$x\neq 0$ the line $\bk\cdot x$, and  every $\DDD_{\E}$-invariant
ideal is homogeneous.
\end{rem}

\ps
\subsection{Linear semigroups}
One would like to extend the lemma above to a statement of the form
that a subvariety $Y \subset X$ is stable under a closed semigroup
$\E \subset \End(X)$ if and only if it is $\DDD_{\E}$-invariant
where $\DDD_{\E}$ is the image of $T_{\id}\E$ in $\VF(X)$. We do not
know if such a result holds in general, but we can prove it for
so-called {\it linear semigroups\/} $\E \subset \End(X)$ which is
sufficient for the applications we have in mind.

If $X \subset V$ is a closed subvariety, then $\End(X) \subset
\Mor(X,V)$. Thus we can form linear combinations of endomorphisms of
$X$, but in general the resulting morphism does not have its image
in $X$.

\begin{defn} A subsemigroup $\E \subset \End(X)$ is called {\it linear} if there is a closed embedding $X \into V$ into a vector space $V$ such that the image of $\E$ in $\Mor(X,V)$ is a linear subspace.
\end{defn}

\begin{thm}\label{main-theorem-1}
Let $X$ be an affine variety, and let $\E \subset \End(X)$ be a
linear semigroup. \be
\item
For any $x \in X$ we have $\E(x) = M(x)$.
\item
The subsets $\E(x)\subset X$ are closed and isomorphic to vector
spaces.
\item
$T_{x}M(x) = T_{x}\E(x) = \DDD_{\E}(x)$ for all $x \in X$. \ee In
particular, a closed subvariety $Y \subset X$ is
$\DDD_{\E}$-invariant if and only if it is $\E$-stable, i.e. it is a
union of $\E$-orbits.
\end{thm}
\begin{proof}
Choose a closed embedding $X \subset V$  such that $\E \subset
\Mor(X,V)$ is a linear subspace. Since the map $\End(V) \to
\Mor(X,V)$ is linear and surjective  there is a linear subspace
$\tilde\E \subset\End(V)$ whose image in $\Mor(X,V)$ is $\E$. In
particular, $X$ is stable under $\tilde\E$ and so  $\DDD_{\tilde\E}
\subset \VF_{X}(V)$. The linearity of the map $\End(V) \to \Mor(X,V)$
implies that the image of  $\DDD_{\tilde\E}$ under $\VF_{X}(V) \to
\VF(X)$ is $\DDD_{\E}$, i.e. $\DDD_{\tilde\E}(x) = \DDD_{\E}(x)$ for
all $x \in X$.

Now we apply  Lemma~\ref{main-lemma} to $\tilde \E$ and find that
$\E(x) = \tilde\E(x) = M(x)$, hence (1) and (2). Moreover, we have
$T_{x}\E(x) = T_{x}\tilde\E(x) = \tilde\E(x) = \DDD_{\tilde\E}(x) =
\DDD_{\E}(x)$, hence (3).

Finally, for a closed subvariety $Y \subset X \subset V$ the
$\DDD_{\E}$-invariance is the same as the
$\DDD_{\tilde\E}$-invariance, and $Y$ is stable under $\tilde \E$ if
and only if it is stable under $\E$. Hence the last claim follows
also from the lemma.
\end{proof}

If $\E \subset \End(X)$ is a linear semigroup we define $d_{\E}(X)
:= \max_{x\in X}\dim\E(x)$. By our theorem above we have $d_{\E}(X)
= d_{\DDD_{\E}}(X)$.

\begin{cor}
Let $X$ be an irreducible variety and $\E \subset \End(X)$ a linear
semigroup. The subset defined by $X':=\{x\in X\mid \dim \E(x) = d_{\E}(X)\}$ is then open
and dense in $X$, and the subsets $\E(x)\cap X'$ for $x \in X'$ form
a partition of $X'$.
\end{cor}
\begin{proof}
The first part is Lemma~\ref{lower-semi-cont.lem}. If $y \in \E(x)$,
then $\E(y)\subset \E(x)$. Since, by the theorem above,  the $\E(x)$
are vector spaces and $\dim \E(x) = \dim \DDD_{\E}(x)$, we have
$\E(y) = \E(x)$ in case $y \in X'$. This proves the second claim.
\end{proof}

\begin{rem}
If an algebraic group $G$ acts on a variety $X$, then every element
$A \in \Lie G$ defines a vector field $\xi_{A}$. It is known that
for a connected group $G$ a closed subvariety $Y \subset X$ is
$G$-stable if and only if $Y$ is $\xi_{A}$-invariant for all $A \in
\Lie G$. A proof can be found in \cite[III.4.4,
Corollary~4.4.7]{Kr2016Algebraic-Transfor}, and the generalization
to actions of connected ind-groups on affine varieties is given in
\cite[Proposition~7.2.6]{FuKr2015On-the-geometry-of}. Our main
theorem above shows that a similar statement holds for linear
semigroups.
\end{rem}

\begin{rem}\label{Schwarz}
Let $G$ be a reductive group acting on an affine variety $X$. Denote
by $\pi\colon X \to X\quot G$ the algebraic quotient, i.e. the
morphism defined by the inclusion $\OOO(X)^G \into \OOO(X)$.  It is
then clear that every $G$-invariant vector field on $X$ induces a vector
field on the quotient $X\quot G$. \name{Schwarz} shows in
\cite{Sc2013Vector-fields-and-} that if the induced map $\VF(X)^G
\to \VF(X\quot G)$ is surjective, then the \name{Luna} strata of the
quotient $X\quot G$ are intrinsic, i.e. they are permuted by all
automorphisms of $X\quot G$. We will prove a similar statement in
Section~\ref{Inv-Luna-strat.subsec} with the methods developed in
this paper.
\end{rem}

\ms
\section{$\VG$-symmetry}
\subsection{$G$-equivariant endomorphisms}
Now consider an action of an algebraic group $G$ on the affine
variety $X$. Then the induced actions of $G$ on the coordinate ring
$\OOO(X)$ and on the vector fields $\VF(X)$ are locally finite and
rational, and the $G$-invariant vector fields $\VFG(X)$ form an
$\OOO(V)^{G}$-module. Note that the (linear) action of $G$ on
$\VF(X)$ is given by $g\xi:= dg \circ \xi \circ g^{-1}$ if we
consider $\xi$ as a section of the tangent bundle. If we regard
$\xi$ as a derivation $\delta$ of $\OOO(X)$, then  $g\delta :=
(g^{*})^{-1}\circ \delta \circ g^{*}$ where $g^{*}\colon \OOO(X) \to
\OOO(X)$ is the comorphism of $g \colon X \to X$.

The action of $G$ on $\End(X)$ by conjugation induces a linear
action on the tangent space $T_{\id}\End(X)$ which we denote by
$g\mapsto \Ad g$. It follows that the canonical map $\Xi\colon
T_{\id}\End(X) \into \VF(X)$ is $G$-equivariant. In fact, one has
the formula
$$
\xi_{\Ad g (A)}(gx) = dg \, \xi_{A}(x) \text{ for } A \in
T_{\id}\End(X), \ g \in  G \text{ and }  x \in X.
$$
This proves the first part of the following lemma.

\begin{lem}\label{G-invariantVF.lem}
We have $\xi(T_{\id}\End_{G}(X)) \subseteq \VF_{G}(X)$ with equality
if $X$ is a vector space $V$ with a linear action of $G$.
\end{lem}
\begin{proof}
It remains to see that for a linear action of $G$ on the vector
space $V$ we have the identification $T_{\id}\End_{G}(V) = (T_{\id}\End(V))^{G}$. But
this is clear, because  $\End(V)$ is a vector space, $\End_{G}(V) =
\End(V)^{G}$  is a linear subspace, and $\Xi\colon T_{\id}\End(V)
\simto \VF(V)$ is a $G$-equivariant linear isomorphism.
\end{proof}

\ps
\subsection{$\VG$-symmetric subvarieties}
We now come to the main notion of this paper, the {\it
$\VG$-symmetry of subvarieties}. This was already discussed in the
introduction.
\begin{defn}
Let $X$ be an affine variety with an action of an algebraic group
$G$. A closed subvariety $Y \subseteq X$ is called {\it
$\VG$-symmetric\/} if $Y$ is $\VF_{G}(X)$-invariant, i.e., $Y$ is
parallel to all $G$-invariant vector fields $\xi$.
\end{defn}

If $V$ is a vector space with a linear action of the algebraic group
$G$, then $\End_{G}(V) \subset \End(V)$ is a linear subspace and, by
Lemma~\ref{G-invariantVF.lem} above,  the image of
$T_{\id}\End_{G}(V)$ in $\VF(V)$ is the subspace $\VF_{G}(V)$ of
$G$-invariant vector fields. Hence Theorem~\ref{main-theorem-1}
implies the following result.

\begin{thm}\label{main-theorem-2}
Let $V$ be a vector space with a linear action of an algebraic group
$G$. Then a closed subvariety $X \subset V$ is $\VG$-symmetric if
and only if it is stable under $\End_{G}(V)$.
\end{thm}

\begin{exa}
Let $V$ be a $G$-module, and assume that $V^{G}=\{0\}$. Define the
{\it null cone}
$$
\NNN_{0}:= \{v \in V \mid f(v) = 0 \text{ for all }f \in
\OOO(V)^{G}\text{ such that }f(0)=0\}.
$$
Then $\NNN_{0}\subset V$ is a closed $\VG$-symmetric subvariety.
\begin{proof}
We have $\OOO(V)=\bk\oplus \mm_{0}$ where $\mm_{0}$ is the maximal
ideal of $0\in V$, and $\NNN_{0}$ is the zero set of $\mm_{0}^{G}$.
Since $V^{G}$ is fixed under every $G$-equivariant endomorphism
$\phi$ of $V$, we get $\phi^{*}(\mm_{0}^{G}) \subset \mm_{0}^{G}$,
and so $\NNN_{0}$ is stable under $\End_{G}(V)$. Now the claim
follows from the theorem above.
\end{proof}
\end{exa}

\ps
\subsection{Stabilizers}
The next result deals with the relation between $\VG$-symmetric
subvarieties and the $G$-action on $X$. We denote by $G_{x}
\subseteq G$ the stabilizer of $x\in X$, and by $M(x)$ the minimal
$\VF_{G}(X)$-symmetric subvariety containing $x$
(Lemma~\ref{D-invariant.lem}(4)).

\begin{lem}
Let $X$ be an affine $G$-variety. \be
\item
If $Y \subset X$ is a $\VG$-symmetric closed subvariety, then $gY
\subset X$ is $\VG$-symmetric for all $g \in G$.
\item
For $x \in X$ we have $\xi(x) \subseteq (T_{x}X)^{G_{x}}$ for all
$\xi\in \VF_{G}(X)$.
\item
For $x \in X$ and $g \in G$ we have  $gM(x)=M(gx)$, and so
$gM(x)=M(x)$ for $g \in G_{x}$. \ee
\end{lem}
\begin{proof}
(1) If $\xi$ is a $G$-invariant vector field, then $dg \,\xi(x) =
\xi(gx)$ for $x \in X$, $g\in G$. This shows that $\xi(y) \in
T_{y}Y$ if and only if $\xi(gy) \in T_{gy}gY$, and the claim
follows.
\par\smallskip
(2) The formula in (1) shows that for a $G$-invariant vector field
$\xi$ we get $dg \,\xi(x) = \xi(x)$ for $g \in G_{x}$. Hence $\xi(x)
\in (T_{x}X)^{G_{x}}$.
\par\smallskip
(3) This follows from the  minimality of $M(x)$.
\end{proof}

In case of a linear action of $G$ on a vector space $V$ we get the
following result.
\begin{prop}
Let $V$ be a  $G$-module. \be
\item For every closed subgroup $H \subset G$ the fixed point set $V^{H}$ is $\VG$-symmetric.
\item For all $v \in V$ we have $M(v)=\End_{G}(V)(v) \subset V^{G_{v}}$.
\ee
\end{prop}
\begin{proof}
(1) It is clear that $V^{H}$ is stable under all $G$-equivariant
endomorphisms, and so the claim follows from
Theorem~\ref{main-theorem-2}.
\par\smallskip
(2) By (1), $V^{G_{v}}$ is $\VG$-symmetric and contains $v$, hence
$M(v) \subset V^{G_{v}}$ by the minimality of $M(v)$.
\end{proof}
\begin{exa}\label{diagonal.exa}
Let $G\to\GL(V)$ be a diagonalizable representation of an algebraic
group $G$. Then, for a generic $v\in V$, we have $\End_{G}(V)(v) =
V$. In particular, $d_{\End_{G}(V)}(V)=\dim V$.

In fact, let $V = \bigoplus_{\chi\in \Omega} V_{\chi}$ be the
decomposition into weight spaces where $\Omega \subset X(G)$ are
those characters $\chi$ of $G$ such that $V_{\chi}:=\{v\in V \mid
gv=\chi(g)\cdot v\}$ is nontrivial. Then $\End_{G}(V)$ contains
$\LLL:=\bigoplus_{\chi\in\Omega}\LLL(V_{\chi})$ where $\LLL(W)$
denotes the linear endomorphisms of the vector space $W$.  It
follows that for any $v = (v_{\chi})_{\chi\in\Omega}$ such that
$v_{\chi}\neq 0$ for all $\chi\in\Omega$ we have $\LLL(V) = V$, thus
the claim.
\end{exa}

\ps
\subsection{Reductive groups}\label{RedGr.subsec}
If $X$ is an affine $G$-variety and $Y \subset X$ a closed and
$G$-stable subvariety, then  $\VF_{Y}(X) \subset \VF(X)$ is a
$G$-submodule and the linear map $\rho\colon\VF_Y(X) \to \VF(Y)$ is
$G$-equivariant. If $Y$ is also $\VG$-symmetric, then
$\rho(\VF_{G}(X)) \subset \VF_{G}(Y)$. But this might be a strict
inclusion, i.e., not every $G$-invariant vector field on $Y$ is
obtained by restricting a $G$-invariant vector field from $X$ (see
Example~\ref{non-liftable-VF.exa} in Section 6). However, if $G$ is
reductive and $X$ is a vector space, then we get $\rho(\VF_{G}(X)) =
\VF_{G}(Y)$. Indeed, $\rho\colon \VF(X) \to \VF(Y)$ is surjective
and, since $G$ is reductive, maps $G$-invariants onto
$G$-invariants. This gives the following result.

\begin{lem}\label{surjectivity-of-VecG.lem}
Let $V$ be a $G$-module and $X \subset V$ a closed $G$-stable and
$\VG$-symmetric subvariety. If a closed subvariety $Y \subset X$ is
$\VG$-symmetric with respect to the action of $G$ on $X$, then it is
also $\VG$-symmetric with respect to the action on $V$. If $G$ is
reductive, then the converse also holds.
\end{lem}

\begin{exa}
Consider the adjoint representation of $\GLn=\GL_{n}(\bk)$ on the
matrices $\M_{n}=\M_{n}(\bk)$. It follows from classical invariant
theory that $\End_{\GL_{n}}(\M_{n})$ is a free module over the
invariants $\OOO(\M_{n})^{\GLn}$, with basis $(p_{i}\colon A \mapsto
A^{i}\mid i=0,\ldots,n-1)$. Note that $p_{0}$ is the constant map $A
\mapsto E$. It follows that the minimal symmetric subspaces $M(A)$
are given by
$$
M(A) = \sum_{i=0}^{n-1}\bk A^{i}.
$$
In particular, a closed subset $Y \subset V$ is $\GL_{n}$-symmetric
if and only if, for any $A \in Y$, the vector space spanned by  all
powers $A^{0}=E, A, A^{2}, \ldots$ is contained in $Y$. Note that
the minimal subsets $M(A) \subset \M_{n}$ are exactly the {\it
commutative unitary subalgebras of $\M_{n}(\bk)$} generated by one
element.

Recall that a matrix $A$ is {\it regular\/} if its centralizer
$(\GLn)_{A}$ has dimension $n$ which is the minimal dimension of a
centralizer.  Equivalently, the minimal polynomial of $A$ coincides
with the characteristic polynomial of $A$. The following is known.
\be
\item
$A$ is regular if and only if $\dim M(A) = n$.
\item
For a regular matrix  $A$ one has $M(A) = (\M_{n})^{(\GL_{n})_{A}}$.
\ee An example of a closed $\VG$-symmetric subvariety is the
nilpotent cone $\NNN\subset M_{n}$ consisting of all nilpotent
matrices. It is also known that for a nilpotent matrix $N$ all
powers $N^{k}$ are contained in the closure of the conjugacy class
$C(N)$ of $N$, as well as their linear combinations. (In fact,
$N':=\sum_{k\geq0}a_{k}N^{k}$ is conjugate to $N$ if $a_{0}\neq 0$,
because $\ker {N'}^{j}=\ker N^{j}$ for all $j$.)  Hence these
closures $\overline{C(N)}$ are $\VG$-symmetric as well.
\end{exa}

In the example above we have $M(A) = (\M_{n})^{(\GLn)_{A}}$ for a
regular matrix $A$. This is an instance of the following general
result which is due to \name{Panyushev}
\cite[Theorem~1]{Pa2002On-covariants-of-r}. For the convenience of
the reader we give a short proof.

\begin{prop}\label{Panyushev.prop}
Let $V$ be a $G$-module where $G$ is reductive. If the closure
$\overline{Gv}$ of the orbit of $v$ is normal and if
$\codim_{\overline{Gv}}(\overline{Gv}\setminus Gv) \geq 2$, then
$M(v) = V^{G_{v}}$.
\end{prop}

\begin{proof}
The assumptions on the orbit closure imply that $\OOO(\overline{Gv})
= \OOO(Gv)$. Let $w \in V^{G_{v}}$. We will show that there is a
$G$-equivariant morphism $\phi\colon V \to V$ such that $\phi(v)=w$.
Since $G_{w}\supseteq G_{v}$ there is a $G$-equivariant morphism
$\mu\colon Gv \to V$ such that $\mu(v)=w$. The comorphism has the
form $\mu^{*}\colon \OOO(V) \to \OOO(Gv) = \OOO(\overline{Gv})$,
hence $\mu$ extends to a morphism $\tilde\mu\colon \overline{Gv} \to
V$ which is again $G$-equivariant. Since $G$ is reductive and
$\overline{Gv}\subset V$ closed and $G$-stable, the morphism
$\tilde\mu$ extends to a $G$-equivariant morphism $\phi\colon V \to
V$ with $\phi(v)=w$.
\end{proof}

\ps
\subsection{Dense orbits}
Let $X$ be an irreducible affine variety, and let $\E \subset
\End(X)$ be a subsemigroup. An interesting question is whether $\E$
has a dense orbit, i.e. whether there exists an $x \in X$ such that
$\overline{\E(x)}=X$.

\begin{lem}\label{DenseOrbit.lem}
Let $\E \subset \End(X)$ be a linear semigroup. Then the following
are equivalent. \be
\item[(i)] $\E$ has a dense orbit in $X$.
\item[(ii)] $d_{\E}(X) = \dim X$.
\item[(iii)] There exists an $x \in X$ such that $\E(x)=X$.
\item[(iv)] One has $\E(x) = X$ for all $x$ in an open dense subset of $X$.
\ee If this holds, then $X$ is a vector space.
\end{lem}
\begin{proof}
If $\E$ is a linear semigroup, then $\E(v)\subset V$ is a linear
subspace and therefore closed in $X$. It is now clear that the first
three statements are equivalent, and (iv) follows from (iii) and the
last statement of Lemma~\ref{main-lemma}.
\end{proof}
\begin{prop}
Let $G$ be a reductive group, and let $V$ be a faithful $G$-module.
\be
\item
If the generic $G$-orbits in $V$ are closed with trivial stabilizer,
then $\End_{G}(V)$ has a dense orbit in $V$, i.e.
$d_{\End_{G}(V)}(V)=\dim V$.
\item
If $G$ is semisimple and $d_{\End_{G}(V)}(V)=\dim V$, then the
generic $G$-orbits in $V$ are closed with trivial stabilizer. \ee
\end{prop}
\begin{proof} Set $\E:=\End_{G}(V)$.
\par\smallskip
(1) If the orbit $Gv$ is closed and $G_{v}$ trivial, then $\E(v) =
V$ by Proposition~\ref{Panyushev.prop}.
\par\smallskip
(2) If $d_{\E}(V)=\dim V$, then, by the lemma above, we have $\E(v)
= V$ for all $v$ from a dense open subset $U \subset V$. Since
$\E(v) \subset V^{G_{v}}$ and since the action is faithful, we see
that $G_{v}$ is trivial for all $v \in U$, i.e. the generic
stabilizer is trivial. Since $G$ is semisimple this implies that the
generic orbits are closed, see
\cite[Corollary~1]{Po1970Criteria-for-the-s}.
\end{proof}
\begin{rem}
Example~\ref{diagonal.exa} shows that the assumption in (2) that $G$
is semisimple is necessary.
\end{rem}

\ps
\subsection{Invariance of the isotropy strata}\label{Inv-Luna-strat.subsec}
Let $G$ be a reductive group acting on an affine variety $X$, and
let $\pi\colon X \to Z:=X\quot G$ denote the algebraic quotient,
i.e. the morphism corresponding to the inclusion $\OOO(X)^G\into
\OOO(X)$. Then $\pi$ sets up a bijection between the closed orbits
of $X$ and the points of $X\quot G$.

Let $x\in X$ be  such that the orbit $Gx$ is closed. Then the
isotropy group $H:=G_x$ is reductive.
The {\it isotropy stratum\/} $Z_H \subseteq Z$ consists of the
closed orbits whose isotropy groups are conjugate to the reductive
subgroup $H \subseteq G$.


If $X$ is a $G$-module $V$ with quotient $Y:=V\quot G$, then the
isotropy strata are locally closed and irreducible. In fact, $Y_H$
is open (and dense) in the closed subset $\pi(V^H)\subset Y$, since
it is equal to $\pi(V^H)\setminus\bigcup_{L\supsetneqq H}Y_L$.

\begin{prop}\label{isotropy-strata.prop}
Assume that the canonical map $p_X\colon \VF_G(X) \to \VF(X\quot G)$
is surjective. Then the isotropy strata of the algebraic quotient
$X\quot G$ are stable under the connected component $\Aut(X\quot
G)^\circ$ of the automorphism group of $X\quot G$.
\end{prop}

\begin{proof}
(1) We first consider the case where $X$ is a $G$-module $V$ with
quotient $Y:=V\quot G$.
It is clear that a $G$-equivariant endomorphism $\phi\colon V \to V$
sends $V^H$ to $V^H$, and thus induces an endomorphism $\bar\phi$ of
the quotient $V\quot G$ such that $\bar\phi(\overline{Y_H})\subseteq
\overline{Y_H}$. It follows that these closures are
$\End_G(V)$-stable, hence invariant under $\VF_G(V)$, by
Proposition~\ref{E&D-invariance.prop} and
Lemma~\ref{G-invariantVF.lem}.

If  the canonical map $p_V\colon \VF_G(V) \to \VF(Y)$ is surjective,
then the closures of the strata $Y_H$ are invariant under $\VF(Y)$. Since
$\Lie\Aut(Y)$ is a Lie-subalgebra of $\VF(Y)$, it follows from
\cite[Proposition~7.2.6]{FuKr2015On-the-geometry-of} that the
closures $\overline{Y_H}$ are stable under $\Aut(Y)^\circ$. Since
the closure of every stratum is a finite union of strata we finally
get that the strata $Z_H$ are $\Aut(Z)^\circ$-stable. \ms (2) In
general, we can assume that $Z$ is a $G$-stable closed subset of a
$G$-module $V$, and so $Z:=X\quot G$ is a closed subset of
$Y:=V\quot G$. By definition, the isotropy strata of $Z$ are the
intersections $Z_H=Y_H\cap Z$. We have the following commutative
diagram where the horizontal maps are the restriction maps of vector
fields to closed subsets:
\[
\begin{CD}
\VF_G(V) @>{\res|^V_X}>{\text{\tiny surjective}}>  \VF_G(X) \\
@V{p_V}VV  @V{p_X}V{\text{\tiny surjective}}V \\
\VF(Y) @>{\res|^Y_Z}>> \VF(Z)
\end{CD}
\]
Since $V$ is a vector space, the restriction map $\VF(V) \to \VF(X)$
is surjective, hence $\res|_X^V$ is also surjective, because $G$ is
reductive. By assumption, $p_X$ is also surjective.

We have seen in (1) that the closures of the  isotropy strata of $Y$
are invariant under the image of $\VF_G(V)$ in $\VF(Y)$. Hence, the
closures of the isotropy strata of $Z$ are invariant under the image
of $\VF_G(V)$ in $\VF(Z)$ which is all of $\VF(Z)$ as we have just
seen. Now the claim follows as in (1).
\end{proof}
\begin{rem}
Proposition \ref{isotropy-strata.prop} is a variant of a stronger result of \name{Schwarz}
\cite{Sc2013Vector-fields-and-} (cf. Remark~\ref{Schwarz}) which
shows that under the same assumptions the (irreducible)
\name{Luna}-strata are permuted under the full automorphism group of
the quotient $X\quot G$. In his proof he shows that the vector
fields span the tangent spaces of the \name{Luna}-strata, and thus
an automorphism has to permute the strata of the same dimension.
\end{rem}
\ms
\section{First integrals}
\subsection{The field of first integrals}
Let $X$ be an irreducible affine variety, and let $\DDD \subset
\VF(X)$ be a linear subspace.
\begin{defn}
A {\it first integral of $\DDD$\/}  is a rational function $f \in
\bk(X)$ with the property that $\xi f = 0$ for all $\xi \in \DDD$.
If $X$ is a $G$-variety and $\DDD := \VF_{G}(X)$, then a first
integral of $\DDD$ will be called a {\it first integral for the
$G$-action on $X$}.
\end{defn}
It is easy to see that the first integrals of $\DDD$ form a subfield
of $\bk(X)$ which we denote by $\FFF_{\DDD}(X)$. If $\DDD =
\VF_{G}(X)$, then we write $\FFF_{G}(X)$ instead of
$\FFF_{\VF_{G}(X)}(X)$.

\ps From now on assume that $X$ is an irreducible affine variety,
and that $\DDD \subset \VF(X)$ is a linear subspace. We want to show
that the first integrals are the rational functions on a certain
``quotient'' of the variety $X$ which will be defined for the action
of a linear semigroup $\E \subseteq \End(X)$ in a similar way as the
quotient for the action of an algebraic group, see
Section~\ref{modE.subsec}.

\begin{lem}\label{first-integral.lem}
Let $f \in \bk(X)$ be a rational function. \be
\item
Assume that there is an open dense $U \subset X$ where $f$ is
defined and has the property that $f$ is constant on $M(x)\cap U$
for all $x \in U$. Then $f$ is a first integral of $\DDD$.
\item
Assume that $f$ is a first integral of $\DDD$. If $f$ is defined in
$x \in X$ and if $T_{x}M(x) = \DDD(x)$, then $f$ is constant on
$M(x)$. \ee
\end{lem}
\begin{proof}
(1) Since $M(x)$ is $\DDD$-invariant we have $\xi(x) \in
T_{x}\,M(x)$ for all $x \in U$ and all $\xi \in \DDD$. Hence we have $(\xi
f)(x) = \xi(x) f = \xi(x) f|_{M(x)\cap U} = 0$, because
$f|_{M(x)\cap U}$ is constant, and so $\xi f = 0$ for all $\xi \in
\DDD$.
\par\smallskip
(2) There is $d\geq 0$ such that $\dim \DDD(y)\leq d$ for all $y \in
M(x)$, with equality on a dense open set $M' \subset M(x)$
(Lemma~\ref{lower-semi-cont.lem}). In particular, $\dim M(x) \leq
\dim T_{x}M(x) = \dim \DDD_{x} \leq d$. On the other hand,
$\DDD_{y}\subset T_{y}M(x)$ for all $y \in M(x)$. We can assume that
$M'$ consists of smooth point of $M(x)$. Then, for every $y \in M'$,
we get $d=\dim \DDD_{y}\leq \dim T_{y}M(x) = \dim M(x)$. Hence $d =
\dim M(x)$, and so $T_{y}M(x)=\DDD(y)$ for all $y\in M'$. Since $f$
is defined in $x$, it is defined in a dense open set $M''\subset
M'$. But then $f|_{M''}$ is constant, because $\delta f=0$ for all
$u \in M''$ and all $\delta \in T_{u}M(x)$.
\end{proof}
\begin{rem}\label{first-integral.rem}
If $\E \subset \End(X)$ is a linear semigroup and $\DDD:=\DDD_{\E}$,
then a rational function $f\in\bk(X)$ defined on an open set $U
\subset X$ is a first integral for $\DDD$ if and only if $f$ is
constant on $\E(x)\cap U$ for all $x \in U$. This follows from the
lemma above, because in this case we have $\E(x) = M(x)$ and
$T_{x}M(x) = \DDD(x)$ for all $x \in X$, by
Theorem~\ref{main-theorem-1}.
\end{rem}

Now  choose a closed embedding $X\subset V$ into a vector space $V$.
We know from Lemma~\ref{lower-semi-cont.lem} that the subset $X':=\{x\in X \mid
\dim \DDD(x)=d_\DDD(x)\}$ is open and dense in $X$. Consider the map
$$
\pi\colon X' \to \Gr_{d_{\DDD}(X)}(V) \text{ given by }\pi(x) :=
\DDD(x)\subset T_{x}X \subset V.
$$
\begin{lem}
The map $\pi\colon X' \to \Gr_{d_\DDD(x)}(V)$ is a morphism of
varieties.
\end{lem}
\begin{proof}
We will use the Pl\"ucker-embedding $\Gr_{d}(V) \into
\PP(\bigwedge^{d}V)$, $d := d_\DDD(x)$. For  $x \in X'$ choose
$\xi_{1},\ldots,\xi_{d} \in \DDD$ such that
$\xi_{1}(x),\ldots,\xi_{d}(x)$ is a basis of $\DDD(x)$. Then
$\DDD(x)= \xi_{1}(x)\wedge
\xi_{2}(x)\wedge\dots\wedge\xi_{d}(x)\in\bigwedge^{d}V$. It follows
that  there is an open neighborhood $U_{x}\subset X'$ of $x$ such
that $\xi_{1}(u),\ldots,\xi_{d}(u)$ is a basis of $\DDD(u)$ for all
$u \in U_{x}$. Since $\pi(u) =
[\xi_{1}(u)\wedge\cdots\wedge\xi_{d}(u)] \in \PP(\bigwedge^{d}V)$ we
see that $\pi|_{U_{x}}$ is a morphism, and  the claim follows.
\end{proof}

\ps
\subsection{The quotient mod $\E$}\label{modE.subsec}
Let $\E \subset \End(V)$ be a linear semigroup, and let $\DDD_{\E}
\subset \VF(V)$ denote the image of $T_{\id}\E = \E$. Let $X \subset
V$ be a closed irreducible $\E$-stable subvariety. Under these
assumptions we have $\E(x)=M(x)=\DDD_{\E}(x) \subset V$  for all $x
\in X$ (Lemma~\ref{main-lemma}). As above, define
$$
X':=\{x\in X\mid \dim \E(x) = d_{\E}(X)\},
$$
and consider the morphism  $\pi\colon X' \to \Gr_{d_\E(x)}(V)$,
$x\mapsto \E(x)\subset T_{x}X \subset V$.

\begin{prop}\label{quotient-mod-E.prop}
\leavevmode
\be
\item
For all $x \in X'$ we have $\pi^{-1}(\pi(x)) = \E(x)\cap X'$.
\item
$\pi$ induces an isomorphism $\pi^{*}\colon \bk(\overline{\pi(X')})
\simto \FFF_{\DDD_{\E}}(X)$.
\item
We have $\tdeg_{\bk}\FFF_{\DDD_{\E}}(X) = \dim X - d_{\E}(X)= \dim
\overline{\pi(X')}$.
\item
$\FFF_{\DDD_{\E}}(X) = \bk$ if and only if $d_{\E}(X) = \dim X$, and
then $X \subset V$ is a linear subspace. \ee
\end{prop}
The proposition shows that the orbits on the open subvariety $X'
\subseteq X$, i.e. the subsets $\E(x)\cap X'$, are disjoint and are
the fibers of the morphism  $\pi\colon X' \to \Gr_{d_\E(x)}(V)$.
Therefore, we will use the notion $X\quot \E$ for the closure
$\overline{\pi(X')}$ and call it  the {\it quotient of $X$} under
the action of the semigroup $\E$ of endomorphisms.
\begin{proof}
(1) For $y\in \E(x)\cap X'$ we have $\E(y)=\E(x)$, hence $\pi(y) =
\pi(x)$. If $y \in X' \setminus \E(x)$, then $\E(y) \neq \E(x)$ and
so $\pi(y)\neq\pi(x)$.
\par\smallskip
(2) By Remark~\ref{first-integral.rem} a rational function $f \in
\bk(X)$ defined on an open set $U\subset X'$ is a first integral if
and only if it is constant on the subsets $\E(x)\cap U$ for all $x
\in U$. We can assume that $\pi(U) \subset \Gr_{d_\E(x)}(V)$ is
locally closed and smooth and that $\pi\colon U \to \pi(U)$ is
smooth.  Then it is a well-known fact that $\pi^{*}(\OOO(\pi(U)))
\subset \OOO(U)$ are the regular functions on $U$ which are constant
on the fibers.
\par\smallskip
(3) This is clear. \ps (4) If $d_{\E}(X) = \dim X$, then $X = \E(x)$
for a generic $x\in X$ (Lemma~\ref{DenseOrbit.lem}), and so $X$ is a
linear subspace of $V$.
\end{proof}
\begin{cor}
If $X \subset V$ is not a linear subspace, then there exist
non-constant first integrals.
\end{cor}
Note that if $X$ is smooth, then it is a linear subspace, because
$X$ is a closed cone, see Remark~\ref{closed-cone.rem}.

\begin{exa}\label{kId.exa}
Let $X \subset V$ be a closed cone, and let $\E:=\bk\cdot\id_{V}
\subset \End(V)$. Then $\E(x) = \bk x$ for all $x \in X$, hence
$X\quot \E = \PP(X)$ and $\FFF_{\DDD_{\E}}(X) = \bk(\PP(X))$.
\end{exa}

\ps
\subsection{The symmetric case}
Assume that $V$ is a representation of an algebraic group and that
$\E := \End_{G}(V)$, hence $\DDD_{\E} = \VF_{G}(V)$. Then, for every
$G$-stable and $\VG$-symmetric closed irreducible subvariety $X
\subset V$, the open subset $X' \subset X$ is $G$-stable and the
morphism $\pi\colon X' \to \Gr_{d_\E(x)}(V)$ is $G$-equivariant. In
particular, $\pi^{*}\colon \bk(\overline{\pi(X')}) \simto
\FFF_{G}(X)$ is a $G$-equivariant isomorphism. It follows that for
any $x\in X'$ we have
$$
G_{\pi(x)} = \Norm_{G}(\E(x))
$$
where $\Norm_{G}(W)$ denotes the normalizer in $G$ of the subspace
$W \subset V$.

\begin{lem}\label{first-integrals.lem}
\leavevmode
\be
\item
For $x \in X'$ we have
$$
\tdeg\FFF_{G}(X) \geq \tdeg\FFF_{G}(X)^{G} + \dim G - \dim
\Norm_{G}(\E(x))
$$
with equality on a dense open set  $U \subset X'$.
\item
If $\FFF_{G}(X)^{G}=\bk$, then $\FFF_{G}(X)$ is $G$-isomorphic to
$\bk(G/\Norm_{G}(\E(x))$ for any $x$ in a dense open set of $X'$.
\ee
\end{lem}
\begin{proof}
(1) By \name{Rosenlicht}'s theorem (see
\cite[Satz~2.2]{Sp1989Aktionen-reduktive}) there is an open dense
$G$-stable subset $O \subset \pi(X')$ which admits a geometric
quotient $q \colon O \to O/G$. In particular, the fibers of $q$  are
$G$-orbits and have all the same dimension. Hence $\tdeg\FFF_{G}(X)
= \dim O = \dim O/G + \dim Gu$ for $u \in O$, and we also have the equality $\bk(O/G) =
\bk(O)^{G}=\FFF_{G}(X)^{G}$. If $u = \pi(x)$, then $G_{u}=
\Norm_{G}(\E(x))$ and so
$$
\tdeg\FFF_{G}(X) = \tdeg\FFF_{G}(X)^{G} + \dim G - \dim
\Norm_{G}(\E(x))
$$
for all $x \in U:=\pi^{-1}(O)$. Since $\dim Gu$ is maximal for $u
\in O$ the claim follows.
\par\smallskip
(2) If $\FFF_{G}(X)^{G}=\bk$, then, as a consequence of
\name{Rosenlicht}'s theorem, $G$ has a dense orbit $Gu$ in
$\overline{\pi(X')}$ and so $\FFF_{G}(X) = \bk(Gu)$. If $u =
\pi(x)$, then $Gu \simeq G/\Norm_{G}(\E(x))$, and the claim follows.
\end{proof}

\begin{rem}
Note that $\FFF_{G}(X)^{G}=\bk$ if and only if $G \E(x)$ is dense in
$X$ for a generic $x \in X$, or, equivalently, $\dim X =  d_\E(x) +
\dim G - \dim \Norm_{G}(\E(x))$ for a generic $x \in X$.
\end{rem}

\begin{exa}\label{exa2}
Consider the adjoint representation of $\GL_{2}$ on $\M_{2}$. Then
$\M_{2}' = \M_{2}\setminus \bk I_{2}$ where
$I_{2}=\left[\begin{smallmatrix}1&0\\0&1\end{smallmatrix}\right]$,
and the morphism $\pi$ is equal to the composition
$$
\pi\colon \M_{2}' \onto (\M_{2}/\bk I_{2}) \setminus \{0\} \onto
\PP(\M_{2}/\bk I_{2}).
$$
Choosing the basis
$\overline{\left[\begin{smallmatrix}0&1\\0&0\end{smallmatrix}\right]}$,
$\overline{\left[\begin{smallmatrix}0&0\\1&0\end{smallmatrix}\right]}$,
$\overline{\left[\begin{smallmatrix}1&\phantom{-}0\\0&-1\end{smallmatrix}\right]}$
of $\M_{2}/\bk I_{2}$, the pullbacks of the coordinate functions are
$b,c,\frac{a-d}{2}$, and so $\FFF_{\GL_{2}}(\M_{2}) =
\bk(\frac{a-d}{b},\frac{a-d}{c})$.
\end{exa}
\begin{exa}
For the adjoint representation of $\GL_{n}$ on $\M_{n}$ we claim
that $\GL_{n}$ has a dense orbit in $\pi(\M_{n}')$. In fact, let $S
\in \M_{n}$ be a generic diagonal matrix. Then the span
$\E(S)=\sum_{i=0}^{n-1}\bk S^{i}$ has dimension $n$, hence it is the
subspace of diagonal matrices, and so $\GL_{n}\E(S) \subset \M_{n}$
is the dense subset of all diagonalizable matrices. Moreover,  the
normalizer of $\E(S)$ is equal to $N$, the normalizer of the
diagonal torus $T \subset \GL_{n}$, and so
$\FFF_{\GL_{n}}(\M_{n})\simeq \bk(\GL_{n}/N)$.
\end{exa}

\begin{exa}
The previous example carries over to the adjoint representation of
an arbitrary semisimple group $G$ on its Lie algebra $\gg:=\Lie G$.
If $s \in \gg$ is a regular semisimple element, then the orbit $Gs$
is closed and the stabilizer of $s$ is a maximal torus $T$. This
implies by Proposition~\ref{Panyushev.prop} that $\E(s) =
\gg^{T}=\Lie T$ which is a Cartan subalgebra of $\gg$. Again, $G
\E(s)\subset \gg$ is the dense set of semisimple elements of $\gg$,
and the normalizer of $\E(s)$ is equal to $N$, the normalizer of $T$
in $G$. Hence $\FFF_{G}(\Lie G) \simeq \bk(G/N)$.
\end{exa}

In the examples above there are no $G$-invariant first integrals:
$\FFF_{G}(X)^{G}=\bk$. This is not always the case as the next two
examples show. However, it holds for a representation of a reductive
group $G$ in case the generic fiber of the quotient map contains a
dense orbit (Proposition~\ref{FGinvariantsAreTrivial.prop}).

\begin{exa} Suppose that $U\subset\GL(V)$ is unipotent and that the generic stabilizer of the action of $U$ on $V$ is trivial.
Then it follows from a result of Domokos \cite[Theorem~1.1,
p.840]{Do2008Covariants-and-the} that $\FFF_{U}(V)\simeq\bk$. In
this example we look at an instance where the generic stabilizer is
not trivial.

Let $U=\left\{\g a b c\mid a,b,c\in\bk\right\}\subset\GL_{3}(\bk)$ be the
unipotent group of upper triangular matrices, and consider the
adjoint representation of $U$ on its Lie algebra $\uu:=\Lie U =
\left\{\vv xyz\mid x,y,z\in\bk\right\}$. For $u=\g abc \in U$ and $v=\vv xyz
\in \uu$ we find
\[\tag{$**$}
\Ad(u) v = u v u^{-1} = \begin{bmatrix} 0 & x & -c x + y + a z\\
&0&z\\&&0\end{bmatrix}
\]
which shows that the fixed points are $\uu^{U}=\bk\vv 010$ and the
other orbits are the parallel lines $\Ad(U)\vv xyz =\vv
x0z+\uu^{U}$. It follows that the invariant ring is given by
$\OOO(\uu)^{U}=\bk[x,z]$. We have an exact sequence of $U$-modules
$$
0\to \uu^{U} \into \uu \overset{p}{\to} \bk^{2}\to 0 \ \text{ where
}\ p\left(\vv xyz\right):=(x,z).
$$
We claim that  the covariants $\E:=\Cov(\uu,\uu)$ are generated as a
$\OOO(\uu)^{U}$-module by $\id_{\uu}$ and the map
$$\phi_{0}\colon  \vv xyz \mapsto \vv 010.$$
This implies that $\E(v) = \bk v +\uu^{U}$ for
$v\in\uu\setminus \uu^{U}$, hence $d_{\E}(\uu) = 2$ and
$\uu'=\uu\smallsetminus \uu^{U}$. It follows that
$$
\uu\quot \E = \PP(\uu/\uu^{U})\simto \PP^{1}.
$$
In particular, the action of $U$ on the quotient is trivial, and so
$$
\FFF_{U}(\uu) = \FFF_{U}(\uu)^{U}=\bk(\frac{x}{z}).
$$
In order to prove the claim, let $\phi\colon \uu \to \uu$ be a
covariant,
$$
\phi\left(\vv xyz\right)=\vv {p(x,y,z)}{q(x,y,z)}{r(x,y,z)}\ \text{ where }\
p,q,r \in \OOO(\uu)=\bk[x,y,z].
$$
Then, by $(**)$, we get for $a,b,c\in \bk$
\[\tag{1}
q(x,-c x + y + a z,z) = -c\cdot p(x,y,z)+q(x,y,z)+a\cdot r(x,y,z).
\]
This shows that $q$ is linear in $y$, i.e. $q(x,y,z) = q_{0}(x,z) +
q_{1}(x,z)y$, and so
\[\tag{2}
q(x,-c x + y + a z,z) = q_{0}(x,y) + q_{1}(x,z)(-c x + y + a z) =
q_{0} -c\cdot q_{1}x+ q_{1}y + a\cdot q_{1}z
\]
Comparing (2) with (1) we get
\[\tag{3}
p= q_{1} x, \quad q = q_{0} + q_{1}y, \quad r = q_{1}z,
\]
hence $\phi = q_{1}\id_{\uu}+ q_{0}\phi_{0}$, as claimed.
\end{exa}

\begin{exa}
Let $G$ be a reductive group and $V$ an irreducible $G$-module. If
the connected component of the center $Z(G)^{0}$ acts nontrivially,
then $\End_{G}(V) =\bk \id_{V}$. Hence, by Example~\ref{kId.exa}, we
get the following equalities $V\quot\End_{G}(V) \simeq \PP(V)$, $\FFF_{G}(V) = \bk(\PP(V))$,
and $\FFF_{G}(V)^{G} = \bk(\PP(V\quot (G,G)))$.
\newline
(In order to see that $\End_{G}(V) =\bk \id_{V}$ we just remark that
the $G$-module $V^{*}$ occurs only once in $\OOO(V)$, namely in
degree 1. In fact, $Z(G)^{0}$ acts on $V$ via a character $\chi$,
and thus via $\chi^{-d}$ on the homogeneous functions $\OOO(V)_{d}$
of degree $d$.)
\end{exa}

This example generalizes to the situation where $V$ is a reducible
$G$-module such that the characters of $Z(G)^{0}$ on the irreducible
components of $V$ are linearly independent.

\begin{exa}\label{Omin.exa}
Let $V$ be an irreducible representation of a reductive group $G$.
For the orbit $\Om \subset V$ of highest weight vectors we have
$\bOm = \Om \cup \{0\}$, and $\bOm$ is normal with rational
singularities (see \cite{He1979The-normality-of-c}). Clearly, $\bOm$
is $\VG$-symmetric, i.e. stable under all $G$-equivariant
endomorphisms of $V$. We claim that $\E:=\End_{G}(\bOm) =
\bk\cdot\id$. In fact, if $v \in V$ is a highest weight vector, then
the $G$-orbit of $[v]\in\PP(V)$ is closed, and thus the stabilizer
$P$ of $[v]$ is a parabolic subgroup. Hence $P$ is the normalizer of
$G_{v}$ in $G$,  and so $P/G_{v}=\kst$. Since,
$\Aut_{G}(\bOm)=\Aut_{G}(\Om)\simeq P/G_{v}=\kst$ the claim follows.

As a consequence we get $\bOm'=\Om$, $\bOm\quot \E = \Om/\kst =
\PP(\bOm) \subset \PP(V)$, and so $\PP(\bOm)$ is the closed orbit of
highest weight vectors in $\PP(V)$. In particular, $\FFF_{G}(\bOm) =
\bk(\PP(\Om))$, and $\FFF_{G}(\bOm)^{G}=\bk$.
\end{exa}

\ps
\subsection{First integrals for reductive groups}
Let $G$ be a reductive group, and let $X$ be an irreducible
$G$-variety. Denote by $q\colon X \to X\quot G$ the quotient. Then
\name{Luna}'s slice theorem (see
\cite[pp.~97--98]{Lu1973Slices-etales}) implies the existence of a
{\it principal isotropy group\/} $H \subset G$. This means the
following: \be
\item If $Gx \subset X$ is a closed orbit, then $G_{x}$ contains a conjugate of $H$.
\item The set $(X\quot G)_{\text{\it pr}}$ of points $\xi\in X\quot G$ such that the closed orbit in the fiber $q^{-1}(\xi)$ is $G$-isomorphic to $G/H$ is open and dense in $X\quot G$.
\ee It follows that every closed orbit contains a fixed point of
$H$, hence $\pi(X^{H}) = X \quot G$.
\par\smallskip
The open dense subset $(X\quot G)_{\text{\rm pr}}$ of $X\quot G$ is
called the {\it principal stratum}, and the closed orbits over the
principal stratum are the {\it principal orbits}. If the action on
$X$ is {\it stable}, i.e. if the generic orbits of $X$ are closed,
then the principal orbits are generic.

\begin{thm}\label{first-integrals.thm}
Let $G$ be reductive, $V$ a $G$-module, and let $X \subset V$ be a
$G$-stable and $\VG$-symmetric irreducible closed subvariety. Assume
that the generic orbit of $X$ is closed, with principal isotropy
group $H \subset G$. Then $\FFF_{G}(X) = \bk(G/N)$ where $N :=
\Norm_{G}(H)$. In particular, $\FFF_{G}(X)^{G}=\bk$.
\end{thm}
\begin{proof}
By assumption, the orbit $Gx$ is principal for a generic $x \in
X^{H}$.  The minimal invariant subset $M(x)$ of $X$ is also
minimal invariant as a subset of $V$
(Lemma~\ref{surjectivity-of-VecG.lem}).  Hence,  $M(x) = V^{H}$ by
Proposition~\ref{Panyushev.prop}. Since $M(x) \subset X^{H}\subset
V^{H}$, we finally get $M(x) = X^{H}=V^{H}$. As we have seen above,
$G X^{H}$ contains all closed orbits, and in particular all $M(y)$
for $y$ in the dense open set of principal orbits. This implies that
$G$ has a dense orbit in $X\quot\End_{G}(X)$. Since the stabilizer
of the image $\pi(V^{H})$ is the normalizer $\Norm_{G}(V^{H})$, it
remains to see $\Norm_{G}(V^{H})=\Norm_{G}(H)$. Since $g(V^{H}) =
V^{gHg^{-1}}$  we get $V^{H}=V^{H\cap gHg^{-1}}$ for any
$g\in\Norm(V^{H})$, hence $H = gHg^{-1}$, because the stabilizer of
a generic elements from $V^{H}$ is $H$.
\end{proof}
Note that for a ``generic'' representation of a semisimple group $G$
the principal isotropy group is trivial, hence there are no
non-constant $G$-invariant first integrals. The irreducible representations of
simple groups with a nontrivial principal isotropy group have been
classified (\cite{AnViEl1967Orbits-of-highest-},
\cite{Po1975Irreducible-simple},  cf.
\cite[\S7]{ViPo1994Invariant-theory}).

The fact that there are no non-constant $G$-invariant first integrals is a
consequence of the following slightly more general result.

\begin{prop}\label{FGinvariantsAreTrivial.prop}
Let $V$ be a representation of a reductive group $G$. Assume that the
generic fiber of the quotient map $q\colon V \to V \quot G$ contains
a dense orbit $O\simeq G/K$, i.e. $\bk(V)^{G}$ is the field of
fractions of $\OOO(V)^{G}$, and that $\codim_{F}F\setminus O\geq 2$.
Then $\FFF_{G}(V) \simeq \bk(G/\Norm_{G}(K))$ and
$\FFF_{G}(V)^{G}=\bk$.
\end{prop}
\begin{proof}
Let $F$ be a fiber of the quotient map $q$ over the principal
stratum, and let $O \subset F$ be the dense orbit. Consider the
morphism $\pi\colon V' \to V\quot\End_{G}(V) \subset \Gr_{d}(V)$,
$d:=d_{\End_{G}(V)}(V)$. We claim that $O\subset V'$, that
$\overline{\pi(O)} = \overline{\pi(V')}=V\quot\End_{G}(V)$, and that
the image of $O$ under $\pi$ is $G/\Norm_{G}(K)$. This will prove
the proposition.

\name{Luna}'s slice theorem tells us that all the fibers of the quotient
map $q$ over the principal stratum are $G$-isomorphic. This implies that
$\End_{G}(V)$ acts transitively on the set of these fibers (see the
argument in the proof of Proposition~\ref{Panyushev.prop}), hence
$\overline{\pi(V')} = \overline{\pi(F\cap V')}$. Since  $F\cap V'$
is open and $G$-stable, we have $O \subset V'$. If $\phi \in
\End_{G}(V)$ and $\phi(v) \in O$ for some $v \in O$, then $\phi(O) =
O$, and so $\phi|_{O}$ is a $G$-equivariant automorphism. On the
other hand, let $\psi\colon O \to O$ be a $G$-equivariant
automorphism. Since $F$ is smooth and the complement of the orbit $O
\subset F$ has codimension $\geq 2$ we have $\OOO(O) = \OOO(F)$.
Therefore, $\psi$ extends to a $G$-equivariant automorphism of $F$,
and then lifts to  a $G$-equivariant endomorphism of $V$.  This
implies that $\End_{G}(V) v \cap O = \Aut_{G}(O)v$. Hence,
$\pi(O)\simeq O/\Aut_{G}(O) \simeq G/\Norm_{G}(K)$, and the claims
follow.
\end{proof}

\section{Actions of $\SLtwo$}\label{SLtwo.sec}
\subsection{Representations}
The standard representation of $\SLtwo$ on $V := \bk^{2}$ defines a
linear action given by
$gf(v):=f(g^{-1} v)$ on the coordinate ring $\OOO(V) = \bk[x,y]$. It is well-known that the homogeneous
components $V_{d}:=\bk[x,y]_{d}$, $d=0,1,2,\ldots$, represent all
irreducible representations of $\SLtwo$, i.e. all simple
$\SLtwo$-modules. As usual, $B \subset \SLtwo$ denotes the
Borel-subgroup of upper triangular matrices, $T \subset \SLtwo$ the
diagonal torus, and $N\subset \SLtwo$ the normalizer of $T$.

\begin{rem}\label{covariant.rem}
An $\SLtwo$-equivariant morphism $\phi\colon V \to W$ between two
$\SLtwo$-modules is called a {\it covariant}. Every covariant is a
sum of homogeneous covariants:
$$
\Cov(V,W)=\bigoplus_{j\in\NN}\Cov(V,W)_{j}.
$$
$\Cov(V,W)$ is a finitely generated $\OOO(V)^{\SLtwo}$-module where
the module structure is given by  $f\phi(v) := f(v)\cdot\phi(v)$
(cf. \cite[IV.~Theorem~2.3.1]{Kr2016Algebraic-Transfor} or
\cite[II.3.2~Zusatz]{Kr1984Geometrische-Metho}). Moreover,
$\End_{\SLtwo}(V) = \Cov(V,V)$.
\end{rem}

\begin{prop} Set $\E_{d}:=\End_{\SLtwo}(V_{d})$, and denote by $J_d:=\OOO(V_d)^{\SLtwo}$ the algebra of invariants.
\be
\item
$\E_{1} = \bk\id_{V_{1}}$, hence $d_{\E_{1}}(V_{1}) = 1$. Moreover,
$V'_{1}=V_{1}\setminus\{0\}$, $V_{1}\quot \E_{1} = \PP(V_{1})$ and
$\FFF_{\SLtwo}(V_{1})\simeq \bk(\SL_{2}/B)$.
\item
$\E_{2} = J_2\id_{V_{2}}$, hence $d_{\E_{2}}(V_{2}) = 1$. Moreover,
$V_{2}' = V_{2}\setminus \{0\}$, $V_{2}\quot \E_{2} = \PP(V_{2})$
and $\FFF_{\SLtwo}(V_{2}) \simeq \bk(\SLtwo/N)$.
\item
$\E_{3} = J_3\id_{V_{3}}\oplus J_3\, dD$ where $D$ is the
discriminant and $dD\colon V_3 \to V_3^*$ its differential. Hence
$d_{\E_{3}}(V_{3}) = 2$. Moreover, $V_3' = V_3 \setminus
\overline{\SLtwo\cdot x^3}$ and $\FFF_{\SLtwo}(V_3) \simeq
\bk(\SLtwo/N)$.
\item
$\E_{4} = J_4\id_{V_{4}}\oplus J_4 H$ where $H$ is the Hessian,
hence $d_{\E_{4}}(V_{4}) = 2$. Moreover, $V_4' = V
\setminus\overline{\SLtwo \cdot \bk x^2y^2}$ and
$\FFF_{\SLtwo}(V_4) \simeq \bk(\SLtwo/O)$ where $O$ is the binary
octahedral group.
\item
For $d\geq 5$, we have $\E_{d}(f) = V_{d}$ for a generic $f \in
V_{d}$, hence $d_{\E_{d}}(V_{d}) = \dim V_{d}$ and
$\FFF_{\SLtwo}(V_{d})=\bk$. \ee
\end{prop}
\begin{proof}
For $d\leq 4$ the $J_d$-module $\E_d=\Cov(V_d,V_d)$ it is a free
module, and the generators can be found in the classical literature,
e.g. in \cite[II.\,\S8]{Sc1968Vorlesungen-uber-I}. For $d=1$ there
is a dense orbit isomorphic to $\SLtwo/U$ whose complement is
$\{0\}$. In particular, $V_1' = V_1\setminus\{0\}$. Since $\E_1 =
\bk\id_{V_1}$ we get $\E_1(f) = \bk f$ and so $V_1\quot \E_1
=\PP(V_1)$. The remaining claims of (1)  follow from
Proposition~\ref{FGinvariantsAreTrivial.prop}.
\par\smallskip
For $d>1$ the generic fibers of the quotient maps $\pi\colon V_d \to
V_d\quot \SLtwo$ are orbits isomorphic to $\SLtwo/H_d$ where $H_2 =
T$, $H_3 = \mu_3$, $H_4 = \tilde D_4$, the binary dihedral group of
order $8$. In the first two cases, the normalizer is equal to $N$,
and we get $V_2' = V_2\setminus\{0\}$ and $V_3'=V_3\setminus
\overline{\SLtwo x^3}$. The remaining claims of (2) and (3)  follow
from Proposition~\ref{FGinvariantsAreTrivial.prop} where in (2) we
use again the fact that $\E_2(f) = \bk f$ for a general $f$ to get
$V_2\quot\E_2 = \PP(V_2)$.
\par\smallskip
For $d=4$ the normalizer of $H_4=\tilde D_4$ is the binary
octahedral group $O$ of order 48 and we have $V_4' = V
\setminus\overline{\SLtwo \cdot \bk x^2y^2}$. Hence
$\FFF_{\SLtwo}(V_4) \simeq \bk(\SLtwo/O)$ by
Proposition~\ref{FGinvariantsAreTrivial.prop}, proving (4).
\par\smallskip
For $d>4$ the stabilizer $H_d$ is trivial for odd $d$ and equal to
the kernel $\{\pm E\}$ of the action for even $d$. Hence, by
Proposition~\ref{Panyushev.prop}, $\E(f) = V_{d}$ for a generic $f$,
and the claims follow.
\end{proof}

%
%
%
%
%

\ps
\subsection{The nullcone $\NNN(V)$}
A very interesting object in this setting is the {\it nullcone}
$\NNN(V) \subset V$ of a representation $V$ of $\SLtwo$ which is
defined in the following way. Denote by $q \colon V \to
V\quot\SLtwo$ the quotient morphism, i.e., $V\quot\SLtwo =
\Spec\OOO(V)^{\SLtwo}$ and $q$ is induced by the inclusion
$\OOO(V)^{\SLtwo} \subset \OOO(V)$. Then $\NNN(V):=q^{-1}(q(0))$, or
equivalently,  $\NNN(V)$ is the zero set of all homogeneous
invariants of positive degree. In case $V = V_{d}$ the elements from
$\NNN(V_{d})$ are classically called {\it nullforms}. One has the
following description. Denote by $T \subset \SLtwo$ the diagonal
torus, and define the {\it weight spaces}
$$
V[i]:=\{f \in V \mid \left[\begin{smallmatrix} t & 0 \\ 0 &
t^{-1}\end{smallmatrix}\right] f = t^{i}f \text{  for all }t \in
\bk^{*}\} \text{ \ for } i \in \NN.
$$
Since the representation of $T$ is completely reducible we have $V =
\bigoplus_{j} V[j]$. For $V = V_{d}$ we get the decomposition $V_{d} =
\bigoplus_{i=0}^{d} V_{d}[d-2i]$, and the weight spaces are
one-dimensional. Note that $\left[\begin{smallmatrix} t & 0 \\ 0 &
t^{-1}\end{smallmatrix}\right] x = t^{-1}x$, and
$\left[\begin{smallmatrix} t & 0 \\ 0 &
t^{-1}\end{smallmatrix}\right] y = t y$, and so
$$
V_{d}[d-2i]=\bk x^{i}y^{d-i}.
$$
\begin{lem}
The following statements for a form $f \in V_{d}$ are equivalent.
\be
\item[(i)] $f$ is a nullform, i.e. $f \in \NNN(V_{d})$.
\item[(ii)] There is a one-parameter subgroup $\lambda\colon \bk^{*} \to \SLtwo$ such that $\lim_{t\to 0}\lambda(t) f = 0$.
\item[(iii)] $f$ is in the $\SLtwo$-orbit of an element from $V_{d}^{+}:=\bigoplus_{i>0} V_{d}[i] \subset V_{d}$.
\item[(iv)] $f$ contains a linear factor with multiplicity $> \frac{d}{2}$.
\ee
\end{lem}
\begin{proof}
(a) The equivalence of (i) and (ii) is a consequence of the famous
\name{Hilbert-Mumford}-Criterion and holds for any representation of
a reductive group.
\par\smallskip
(b) (ii) and (iii) are equivalent, because every one-parameter
subgroup of $\SLtwo$ is conjugate to a one-parameter subgroup of
$T$. This holds for any representation of $\SLtwo$.
\par\smallskip
(c) The equivalence of (iii) and (iv) is clear, because $V_{d}^{+}$
are the forms which contain $y$ with multiplicity at least $\frac{d}{2}$.
\end{proof}

Let $V$ be a representation of $\SLtwo$.  If $\phi\in
\End_{\SLtwo}(V)$ is homogeneous of degree $k$, then $\phi(V[j])
\subset V[kj]$. It follows that $\phi(\bigoplus_{j\geq j_{0}}V[j])
\subset \bigoplus_{j\geq k j_{0}}V[j]$. In particular, the
subspaces $\bigoplus_{j\geq j_{0}}V[j]$ are $\SLtwo$-symmetric for
any $j_{0}\geq 0$, because any endomorphism is a sum of homogeneous
endomorphisms (Remark~\ref{covariant.rem}). Since every element $f
\in \NNN(V)$  is $\SLtwo$-equivalent to an element from
$V^{+}:=\bigoplus_{j>0}V[j]$ it suffices to study the
$\SLtwo$-symmetric subspaces of $V^{+}$.

\ps
\subsection{Special covariants}
For the study of the $\SLtwo$-symmetric subspaces of the nullforms
$\NNN(V_d)$ we need the existence (and the non-vanishing) of certain
covariants which we are going to construct now.

Let $\phi\colon V_{d} \to \End(V_{d})$ and $\psi\colon V_{d} \to
V_{d}$ be homogeneous covariants. Then we define covariants denoted 
$\Phi_{s}=\Phi_{s}(\phi,\psi)\in \End_{\SLtwo}(V_{d})$ by
$$
\Phi_{s}(\phi,\psi)f :=
\phi(f)^{s}\psi(f)=(\phi(f)\circ\phi(f)\circ\cdots\circ\phi(f))(\psi(f))
$$
This is a homogeneous covariant of degree $\deg\Phi_{s}= s\deg\phi +
\deg \psi$.

Let $\sltwo := \Lie \SLtwo$ be the Lie algebra of $\SLtwo$ which
acts on a representation  $V$ of $\SLtwo$ by the adjoint
representation $\ad\colon \sltwo \to \End(V)$. As an $\SLtwo$-module
we have $\sltwo \simto V_{2}$, and $\sltwo[2] = \bk\nn$.

\begin{lem}\label{covariants.lem}
Let $V_{d}$ denote the binary forms of degree $d$, considered as a
representation of $\SLtwo$. \be
\item If $d$ is odd, then there is a quadratic covariant $\phi_{0}\colon V_{d}\to\sltwo$ such that
$\phi_{0}(V_{d}[1]) = \sltwo[2]=\bk\nn$.
\item If $d$ is even, then there is a quadratic covariant $\phi_{0}\colon V_{d} \to \sltwo \otimes \sltwo$ such that $\phi_{0}(V_{d}[2]) = \sltwo[2]\otimes\sltwo[2]$.
\item  If $d\equiv 0\mod 4$, then there is a quadratic covariant $\psi\colon V_{d} \to V_{d}$ such that $\psi(V_{d}[2]) =V_{d}[4]$.
\item If  $d\equiv 2\mod 4$ and $d\geq 10$, then there is a homogeneous covariant $\psi\colon V_{d} \to V_{d}$ of degree 4 such that $\psi(V_{d}[2]) = V_{d}[8]$, and there is no quadratic covariant.
\ee
\end{lem}

For the proof let us recall the \name{Clebsch-Gordan}-decomposition
of the tensor product $V_{d}\otimes V_{e}$ as an $\SLtwo$-module
where we assume that $d \geq e$:
$$
V_{d}\otimes V_{e} \simeq \bigoplus_{r=0}^{e} V_{d+e-2r}.
$$
The projection $V_{d}\otimes V_{e} \to V_{d+e-2r}$ is classically
called the {\it $r$th transvection}. It is given by the following
formula:
\[\tag{T1}\label{transvec1.form}
f\otimes h \mapsto (f,h)_{r} := \sum_{i=0}^{r}(-1)^{i}\binom{r}{i}
\frac{\partial^{r}f}{\partial x^{r-i}
\partial y^{i}} \frac{\partial^{r}h}{\partial x^{i}\partial y^{r-i}}.
\]
The second symmetric power $S^{2}(V_{d})$ has the decomposition
$$
S^{2}(V_{d}) \simto V_{2d}\oplus V_{2d-4}\oplus V_{2d-8}\oplus
\cdots.
$$
Therefore, the quadratic covariants $\tau_r\colon V_d \to
V_{2d-2r}$, $f\mapsto (f,f)_{r}$, are non-zero only for even $r$,
and they are given by
\[\tag{T2}\label{transvec2.form}
\tau_r(f) = (f,f)_{r} = \sum_{i=0}^{r}(-1)^{i}\binom{r}{i}
\frac{\partial^{r}f}{\partial x^{r-i}
\partial y^{i}} \frac{\partial^{r}f}{\partial x^{i}\partial y^{r-i}} \in V_{2d-2r}.
\]
For the non-vanishing of the covariant $\tau_r$ on the nullforms the
following lemma is crucial.

\begin{lem}\label{nonzero-transvection.lem}
For $d =2m$ and an even $r=2s < d$, the transvection
$$
\tau_{2s}(x^{m-1}y^{m+1}) = (x^{m-1}y^{m+1},x^{m-1}y^{m+1})_{2s} \in
V_{2d-4s}
$$
is equal to $c_{m,2s}\cdot x^{d-2s-2}y^{d-2s+2}$ where
$$
c_{m,2s} =
(-1)^{s}(2s)!(s!)^{2}\binom{m-1}{s}\binom{m+1}{s}\binom{2m-s}{s}
\neq 0.
$$
\end{lem}
For the proof we will need some properties of the hypergeometric
function $\Ftt(a_1,a_2,a_3;b_1,b_2;z)$ which we discuss in the
following Section~\ref{hyper.subsec}. The proof of the lemma is then
given in Section~\ref{proof.subsec}.

\begin{proof}[Proof of Lemma~\ref{covariants.lem}]
As above, $\tau_r\colon V_d \to V_{2d-2r}$ denotes the quadratic
covariant $f \mapsto (f,f)_r$ which is nonzero only for even $r$.
\par\smallskip
(a) If $d = 2m+1$, then $\tau_{2m}\colon V_{d} \to V_{2}\simeq
\sltwo$, and  $\tau_{2m}(x^{m}y^{m+1})$ is a non-zero multiple of
$y^{2}$. In fact, for $r=2m$, the sum (\ref{transvec2.form}) has a
single term, namely for $i=m$. This proves (1).
\par\smallskip
(b) Now assume that $d$ is even, $d = 2m$. Then $\tau_{2m-2}\colon
V_{d} \to V_{4}$ has the property that $\tau_{2m-2}(x^{m-1}y^{m+1})$
is a non-zero multiple of $y^{4}\in V_4[4]$. In fact, the sum
(\ref{transvec2.form}) has a single term, namely for $i=m-1$. Since
$\sltwo \otimes \sltwo \simeq V_{0}\oplus V_{2}\oplus V_{4}$ and
$(\sltwo\otimes\sltwo)[4] = \sltwo[2]\otimes \sltwo[2]\simeq
V_4[4]$, we thus get $\phi_{0}\colon V_{d}\to
\sltwo\otimes\sltwo$ a covariant with the property claimed in (2).
\par\smallskip
(c) If $d=2m$ and $m$ even, then, by Lemma~\ref{nonzero-transvection.lem}, $\tau_{m}\colon V_{d} \to V_{d}$ is a quadratic covariant such that $\tau_m(V_{d}[2])=V_{d}[4]$, proving (3). 
\par\smallskip
(d) Finally, if $d=2m$ and $m=2k+1$ is odd, then there is no
quadratic covariant, because $V_{d}$ does not appear in the
decomposition of $S^{2}(V_{d})$. But, for $m\geq 5$, there is a
homogeneous covariant $\psi$ of degree 4 with the required property.
For even $k$ we take
$$
\psi\colon V_{d} \to V_{d}, \ f\mapsto
((f,f)_{3k},(f,f)_{3k+2})_{1},
$$
and for odd $k$
$$
\psi\colon V_{d} \to V_{d}, \ f\mapsto
((f,f)_{3k-1},(f,f)_{3k+3})_{1}.
$$
By Lemma~\ref{nonzero-transvection.lem},
$(x^{m-1}y^{m+1},x^{m-1}y^{m+1})_{r}$ is a nonzero multiple of
$x^{2m-r-2}y^{2m-r+2}$ for every even $r<2m$. It remains to see that
the  transvections $(x^{k}y^{k+4},x^{k-2}y^{k+2})_{1}$ for even $k$
and $(x^{k+1}y^{k+5},x^{k-3}y^{k+1})_{1}$  for odd $k$ are nonzero.
This follows from the transvection formula~(\ref{transvec1.form})
above which gives
\begin{gather*}
(x^{k}y^{k+4},x^{k-2}y^{k+2})_{1}= 8\cdot x^{2k-3}y^{2k+5} = 8 \cdot x^{m-4}y^{m+4},\\
(x^{k+1}y^{k+5},x^{k-3}y^{k+1})_{1}= 16\cdot x^{2k-3}y^{2k+5} = 16
\cdot x^{m-4}y^{m+4}.
\end{gather*}
This proves (4).
\end{proof}

\ps
\subsection{The hypergeometric function $\Ftt$}\label{hyper.subsec}
The \name{Pochhammer} function $(z)_n:= z(z+1)\cdots(z+n-1)$ is
defined for $z \in \CC$ and any integer $n\geq0$ where we set
$(z)_0=1$. Note that $(z)_n=0$ if $z$ is a negative integer $>-n$.
The {\it hypergeometric function\/} $\Ftt(a_1,a_2,a_3;b_1,b_2;z)$ is
defined by the following convergent series
$$
\Ftt(a_1,a_2,a_3;b_1,b_2;z) :=\sum_{k=0}^{\infty}
\frac{(a_1)_k(a_2)_k(a_3)_k}{(b_1)_k(b_2)_k} \frac{z^k}{k!}
$$
where $a_1,a_2,a_3,b_1,b_2 \in \CC$ and $b_1,b_2 \notin
\{0,-1,-2,-3,\ldots\}$, see \cite{Sl1966Generalized-Hyperg}. The
$\Ftt$-series can be evaluated by means of \name{Dixon}'s summation
formula (see \cite[formula~2.3.3.6 on
page~52]{Sl1966Generalized-Hyperg}):
\[\tag{$*$}
\Ftt(a,b,-n;1+a-b,1+a+n;1)
=\frac{(1+a)_{n}(1+\frac{a}{2}-b)_{n}}{(1+\frac{a}{2})_{n}(1+a-b)_{n}},
\]
where $n$ is a nonnegative integer. As mentioned above, the series
is well-defined if neither $1+a-b$ nor $1+a+n$ belong to $ \{0,-1,-2,-3,\ldots\}$. The
right-hand-side is a rational function in $a,b$, namely a quotient
of products of linear terms, and there is some cancellation in the
quotient $\frac{(1+a)_{n}}{(1+\frac{a}{2})_n}$ if $n>1$, e.g.
$\frac{(1+a)_{2}}{(1+\frac{a}{2})_2}=
\frac{(1+a)(2+a)}{(1+\frac{a}{2})(2+\frac{a}{2})}=
\frac{2(1+a)}{2+\frac{a}{2}}$. More precisely, setting $a=2z$, we
find
\begin{align*}
\frac{(1+2z)_{n}}{(1+z)_{n}} &= \frac{(2z+1)(2z+2)\cdots(2z+n)}{(z+1)(z+2)\cdots(z+n)} = \\
&= \begin{cases}
\displaystyle\frac{(2z+1)(2z+3)\cdots(2z+2\ell+1)2^{\ell}}{(z+\ell+1)(z+\ell+2)\cdots(z+2\ell+1)}&\text{ if $n=2\ell+1$ is odd,}\\
\displaystyle\frac{(2z+1)(2z+3)\cdots(2z+2\ell-1)2^{\ell}}{(z+\ell+1)(z+\ell
+2)\cdots(z+2\ell)} &\text{ if $n = 2\ell$ is even.}
\end{cases}
\end{align*}
This shows that the poles of the right hand side of $(*)$ are the
even integers $a$ such that $-\ell-1\geq \frac{a}{2} \geq -n$. But
this implies that $1+a+n$ is a negative integer, and these values
are excluded in the definition of $\Ftt$.

\ps
\subsection{Proof of Lemma~\ref{nonzero-transvection.lem}}\label{proof.subsec}
The following proof was communicated to us by \name{Christian
Krattenthaler}. From formula \eqref{transvec2.form} we get
\[\tag{$**$}
c_{m,r}=\sum_{i=0}^{r}(-1)^i
\binom{r}{i}(m-r+i)_{r-i}\,(m-r+i+2)_{r-i}\, (m-i+2)_i\,(m-i)_i,
\]
It follows that for a fixed integer $r\geq 0$ the coefficient
$c_{m,r}$ is a polynomial in $m$, and the same holds for the claimed
expression of $c_{m,2s}$ given in
Lemma~\ref{nonzero-transvection.lem} above. Therefore, for a given
$r=2s$,  it suffices to prove the equality for infinitely many $m$.
We will do this for all integers $m\geq r+1$ what we assume from now
on.

Using the following obvious identities for the \name{Pochhammer}
function $(z)_n$
\begin{gather*}
\frac{(m-r+i)_{r-i}}{(m-r)_r} = \frac{1}{(m-r)_i},\qquad\frac{(m-r+i+2)_{r-i}}{(m-r+2)_r} = \frac{1}{(m-r+2)_i},\\
(m-i)_i = (-1)^i(-m+1)_i\, \qquad (m-i+2)_i = (-1)^i(-m-1)_i, \\
\binom ri = \frac{1}{i!} (r-i+1)_i = (-1)^i\frac{1}{i!} (-r)_i,
\end{gather*}
we find for the summands in $(**)$
$$
(-1)^i \binom{r}{i} (m-r+i)_{r-i} (m-r+i+2)_{r-i} (m-i+2)_i (m-i)_i 
=  (m-r)_r (m-r+2)_r \, \frac{1}{i!} \frac{(-r)_i (-m-1)_i
(-m+1)_i}{(m-r+2)_i (m-r)_i},
$$
hence
$$
c_{m,r} = (m-r)_r (m-r+2)_r \;\Ftt(-r,-m-1,-m+1;m-r+2,m-r;1).
$$
The $\Ftt$-series can be evaluated by means of \name{Dixon}'s
summation formula $(*)$ above where
$a=-r$, $b = -m-1$, $n=m-1$:
\[\tag{$*\!*\!*$}
\Ftt(-r,-m-1,-m+1;m-r+2,m-r;1) =
\frac{(1-r)_{m-1}(m+2-r/2)_{m-1}}{(1-r/2)_{m-1}(m+2-r)_{m-1}}.
\]
As we have seen above this equality holds for a positive integer
$m\geq 1$ and any $r\in\CC$ as long as $m-r$ is not a negative
integer.
Setting $r =2s$ we get (see the calculation in
Section~\ref{hyper.subsec} above):
\begin{align*}
\frac{(1-2s)_{m-1}}{(1-s)_{m-1}} &= \frac{(2s-1)(2s-2)\cdots(2s-(m-1))}{(s-1)(s-2)\cdots(s-(m-1))} = \\
&= \begin{cases}
\displaystyle\frac{(2s-1)(2s-3)\cdots(2s-2\ell+1)2^{\ell-1}}{(s-\ell)(s-\ell-1)\cdots(s-2\ell+1)}&\text{ if $m=2\ell$ is even,}\\
\displaystyle\frac{(2s-1)(2s-3)\cdots(2s-2\ell+3)2^{\ell-1}}{(s-\ell)(s-\ell
-1)\cdots(s-2\ell+2)} &\text{ if $m = 2\ell -1$ is odd.}
\end{cases}
\end{align*}
Hence, this fraction is well-defined in the given range $r=2s\leq
m-1$, since this means that $s \leq \ell-1$ in both cases, and so
all factors in the denominator are strictly negative integers. In
both cases the denominator can be written as
$(-1)^\ell\frac{(m-s-1)!}{(\ell-s-1)!}$. For the numerator, we find
in case $m=2\ell$:
\begin{multline*}
2^{\ell-1}(2s-1)(2s-3)\cdots 3\cdot 1 \cdot (-1)\cdot (-3)\cdots(-(2\ell-2s-1))=\\
=2^{\ell - 1}\frac{(2s)!}{2^s\cdot
s!}\cdot(-1)^{\ell-s}\frac{(2\ell-2s)!}{2^{\ell-s}(l-s)!} =
\frac{(-1)^{\ell-s} (2s)!(2\ell-2s)!}{2 \,s!(\ell-s)!},
\end{multline*}
and for $m = 2\ell -1$:
\begin{multline*}
2^{\ell-1}(2s-1)(2s-3)\cdots 3\cdot 1 \cdot (-1)\cdot (-3)\cdots(-(2\ell-2s-3))=\\
=2^{\ell-1}\frac{(2s)!}{2^s\cdot
s!}\cdot(-1)^{\ell-s}\frac{(2\ell-2s-2)!}{2^{\ell-s-1}(l-s-1)!} =
\frac{(-1)^{\ell-s} (2s)!(2\ell-2s-2)!}{s!(\ell-s-1)!}.
\end{multline*}
This gives for the right hand side of $(*\!*\!*)$ for an even
$m=2\ell$
$$
(-1)^s\frac{ (2s)!(2\ell-2s)!}{2
\,s!(\ell-s)!}\cdot\frac{(\ell-s-1)!}{(m-s-1)!} =
(-1)^s\frac{(2s)!(m-2s-1)!}{s!(m-s-1)!},
$$
and the same for an odd $m=2\ell-1$
$$
(-1)^s\frac{(2s)!(2\ell-2s-2)!}{s!(\ell-s-1)!}\cdot
\frac{(\ell-s-1)!}{(m-s-1)!} =
(-1)^s\frac{(2s)!(m-2s-1)!}{s!(m-s-1)!}.
$$
The remaining factors are
\begin{align*}
(m-r+2)_r (m-r)_r \frac{(m+2-r/2)_{m-1}}{(m+2-r)_{m-1}} &=
\frac{(m+1)!}{(m-2s+1)!}\frac{(m-1)!}{(m-2s-1)!}\frac{(2m-s)!}{(m-s+1)!}\frac{(m-2s+1)!}{(2m-2s)!} \\
&=\frac{(m+1)!(m-1)!(2m-s)!}{(m-2s-1)!(m-s+1)!(2m-2s)!}.
\end{align*}
Hence
\begin{align*}
c_{m,2s} &= (-1)^s\frac{(2s)!(m-2s-1)!}{s!(m-s-1)!}\cdot\frac{(m+1)!(m-1)!(2m-s)!}{(m-2s-1)!(m-s-1)!(2m-2s)!}\\
&= (-1)^s\frac{(2s)!(m+1)!(m-1)!(2m-s)!}{s!(m-s-1)!(m-s+1)!(2m-2s)!}\\
&=(-1)^{s}(2s)!(s!)^{2}\binom{m-1}{s}\binom{m+1}{s}\binom{2m-s}{s},
\end{align*}
as claimed. \qed

\ps
\subsection{$\VSL$-symmetric subspaces of the nullforms}
We will now determine the minimal $\VSL$-symmetric subspaces of the
nullforms $\NNN(V_{d})$ and calculate the first integrals.

\begin{prop}\label{thm3}
Let $d =2m+1$ be odd, $d\geq 3$. \be
\item $d(\NNN(V_{d})) = m$.
\item $V_{d}^{+}$ is a minimal $\SLtwo$-symmetric subspace of $\NNN(V_{d})$ of dimension $m$.
\item If $M \subset \NNN(V_{d})$ is a minimal $\SLtwo$-symmetric subspace of dimension $m$, then $M = gV_{d}^{+}$ for some $g \in\SLtwo$.
\item $\NNN(V_{d})\quot \End_{\SLtwo}(V_{d}) \simeq \SLtwo/B \simeq \PP^{1}$.
\item $\FFF_{\SLtwo}(\NNN(V_{d})) \simeq \bk(\SLtwo/B)$, in particular $\FFF_{\SLtwo}(\NNN(V_{d}))^{\SLtwo} =\bk$.
\ee
\end{prop}
\begin{proof}
(a) Consider the covariants $\Phi_{s}(\phi,\id)\colon V_{d}\to
V_{d}$ defined above where $\phi$ is the composition
$$
\begin{CD}
\phi \colon V_{d} @>{\phi_{0}}>> \sltwo @>{\ad}>> \End(V_{d})
\end{CD}
$$
and $\phi_{0}\colon V_{d}\to \sltwo$ is from
Lemma~\ref{covariants.lem}(1). By construction, we get
$$
\Phi_{s}(V_{d}[1]) = \ad\nn^{s}V_{d}[1] = V_{d}[2s+1].
$$
This shows that $\End_{\SLtwo}(V_{d}) (V_{d}[1] )= V_{d}^{+}$, hence
(1) and (2).
\par\smallskip
(b) Let  $M = M(f)$ be of dimension $m$. There is a $g\in\SLtwo$
such that $gf \in V_{d}^{+}$, hence $gM(f) = M(gf)$ is contained in 
$V_{d}^{+}$. Since $\dim M(f) = m$ we get $gM(f)  = V_{d}^{+}$. This
gives (3) and shows that $\SLtwo$ acts transitively on the subspaces
$M(f)\subset \NNN(V_{d})$ of dimension $m$, thus on the image of
$\pi\colon \NNN(V_{d}) \to \Gr_{m}(V_{d})$. Since the normalizer of
$V_{d}^{+}$ is $B$, we finally get (4) and (5).
\end{proof}

\begin{prop}\label{d=4m.thm}
Let $d = 2m$ and $m$ even. \be
\item $d(\NNN(V_{d})) = m$.
\item $V_{d}^{+}$ is a minimal $\SLtwo$-symmetric subspace of $\NNN(V_{d})$ of dimension $m$.
\item If $M \subset \NNN(V_{d})$ is a minimal $\SLtwo$-symmetric subspace of dimension $m$, then $M = gV_{d}^{+}$ for some $g \in\SLtwo$.
\item $\NNN(V_{d})\quot \End_{\SLtwo}(V_{d}) \simeq \SLtwo/B \simeq \PP^{1}$.
\item $\FFF_{\SLtwo}(\NNN(V_{d})) \simeq \bk(\SLtwo/B)$, in particular $\FFF_{\SLtwo}(\NNN(V_{d}))^{\SLtwo} =\bk$.
\ee
\end{prop}
\begin{proof}
Define the following covariant
$$
\begin{CD}
\phi \colon V_{d} @>{\phi_{0}}>> \sltwo\otimes\sltwo @>{\alpha}>>
\End(V_{d})
\end{CD}
$$
where $\phi_{0}$ is from Lemma~\ref{covariants.lem}(2), and $\alpha$
is the linear $\SLtwo$-equivariant map $A\otimes B \mapsto \ad A
\circ \ad B$. Then the covariants $\Phi_{s}(\phi,\id)\colon V_{d}\to
V_{d}$ satisfy $\Phi_{s}(V_{d}[2]) = (\ad\nn)^{2s} V_{d}[2]=V_{d}[4s
+ 2]$, and for the covariants $\Phi_{s}(\phi,\psi)$ where $\psi$ is
from Lemma~\ref{covariants.lem}(3) we get $\Phi_{s}(V_{d}[2]) =
\ad\nn^{2s}V_{d}[4] = V_{d}[4s+4]$. As a consequence, we get
$\End_{\SLtwo}(V_{d})(V_{d}[2]) = V_{d}^{+}$, hence (1) and (2). The
remaining claims follow as in the proof of Proposition~\ref{thm3}.
\end{proof}

If $d = 2m$ and $m$ odd we define $V_{d}^{++}:=V_{d}[2] \oplus
V_{d}[6] \oplus V_{d}[8] \oplus \cdots$.
\begin{prop} Let $d =2m$ and $m$ odd, $m\geq 3$.
\be
\item $d(\NNN(V_{d})) = m-1$.
\item $V_{d}^{++}$ is a minimal $\SLtwo$-symmetric subspace of $\NNN(V_{d})$ of dimension $m-1$.
\item If $M \subset \NNN(V_{d})$ is a minimal $\SLtwo$-symmetric subspace of dimension $m-1$, then $M = gV_{d}^{++}$ for some $g \in\SLtwo$.
\item $\NNN(V_{d})\quot \End_{\SLtwo}(V_{d}) \simeq \SLtwo/T$.
\item $\FFF_{\SLtwo}(\NNN(V_{d})) \simeq \bk(\SLtwo/T)$, in particular $\FFF_{\SLtwo}(\NNN(V_{d}))^{\SLtwo} =\bk$.
\ee
\end{prop}
\begin{proof}
(a) We first remark that there is no quadratic covariant $\phi\colon
V_{d}\to V_{d}$, and so $V_{d}^{++}$ is stable under
$\E:=\End_{\SLtwo}(V_{d})$. Now we use the covariants
$\Phi_{s}(\phi,\id)$, as in the proof of the previous proposition, to
show that $\E (V_{d}[2]) \supset V_{d}[4s+2]$. Moreover, the
covariants $\Phi_{s}(\phi,\psi)$ with $\psi$ from
Lemma~\ref{covariants.lem}(4) imply that the inclusion $\E(V_{d}[2]) \supset V_{d}[4s+8]$ holds. It follows that $\E (V_{d}[2]) = V_{d}^{++}$, hence
(1) and (2).
\par\smallskip
(b) Using again that there are no quadratic covariants, we see that
$$\E (V_{d}[4]) \subseteq V_{d}[4] \oplus V_{d}[8] \oplus
V_{d}[10]\oplus\cdots,$$ hence $\dim \E (V_{d}[4]) \leq m-2$.
Therefore, $V_{d}^{++}$ is the only minimal $\SLtwo$-symmetric
subspace of $V_{d}^{+}$ of dimension $m-1$. Now the remaining claims
follow as before, using that the normalizer of $V_{d}^{++}$ is $T$.
\end{proof}

\begin{exa}
The minimal orbit $O_{0} \subset V_{d}$ is the orbit of $y^{d}$.
Denote by $O_{1}$ the orbit of $xy^{d-1}$. Then $X:=\overline{O_{1}}
= O_{1}\cup O_{0}\cup \{0\}$. We claim that $X$ is
$\SLtwo$-symmetric and that $\End_{\SLtwo}(X) = \bk\cdot\id$ in case
$d \geq 5$. In fact, the image of $xy^{d-1}\in V_{d}^{+}$ under a
homogeneous $\phi\in\End_{\SLtwo}(X)$ is again a weight vector of
positive weight, hence a multiple of some $x^{\ell}y^{d-\ell}$ where
$\ell < d- \ell$. Since the stabilizer of $x^{\ell}y^{d-\ell}$ in
$\SLtwo$ is cyclic of order $d-2\ell$ for $\ell < d-\ell$, we see
that $\phi(xy^{d-1})$ is a multiple of $xy^{d-1}$ if $d>4$. (For
$d=4$, $X$ is the nullcone $\NNN(V_{4})$, and the quadratic
covariant $\phi$ sends $O_{1}$ onto $O_{0}$, see
Proposition~\ref{d=4m.thm}.) This implies that $\phi|_{O_{1}} =
\lambda\cdot \id$ for some $\lambda \in\bk$, hence $\phi|_{X} =
\lambda\cdot \id$. As a consequence, $X'=X \setminus\{0\}$, and
$X\quot \E = \PP(X) \subset \PP(V_{d})$.
\end{exa}

\begin{exa}\label{non-liftable-VF.exa}
Let $d =2m$ be even and consider $V_{d}^{+}$ as a $B$-module. It is
not difficult to see that there is always a $B$-covariant $\phi$ of
degree 2. E.g. for $d=6$ it is given by
$$
\phi(a_{1}\cdot x^{2}y^{4}+a_{2}\cdot xy^{5}+a_{3}\cdot y^{6}) =
2a_{1}^{2}\cdot xy^{5} + a_{1}a_{2}\cdot y^{6}.
$$
On the other hand, for $d=2m\geq 6$ and $m$ odd there is no
$\SLtwo$-covariant of $V_{d}$ of degree 2
(Lemma~\ref{covariants.lem}(4)). Since
$\End_{\SLtwo}(V)=\End_{B}(V)$ for every $\SLtwo$-module $V$, we see
that for $d\equiv 2\mod 4$ and $d\geq 6$ the restriction map $\End_{B}(V_{d}) \to \End_{B}(V_{d}^{+})$ is not surjective.
\end{exa}
\addtocounter{section}{1} \setcounter{subsection}{0}

\ms
\section*{Appendix: Ind-varieties and ind-semigroups}
\addcontentsline{toc}{section}{Appendix: Ind-varieties and
ind-semigroups}\refstepcounter{section}\renewcommand{\thesection}{A}

An introduction to ind-varieties and ind-groups can be found in
\name{Kumar}'s book \cite[Chapter IV]{Ku2002Kac-Moody-groups-t}.
\subsection{Basic definitions}
The following is borrowed from \cite{FuKr2015On-the-geometry-of}.

\begin{defn}\label{indvar.def}
An {\it ind-variety} $\VVV$ is a set together with an ascending
filtration $\VVV_{0}\subset \VVV_{1}\subset \VVV_{2}\subset
\cdots\subset \VVV$ such that the following holds: \be
\item $\VVV = \bigcup_{k \in \NN}\VVV_{k}$;
\item Each $\VVV_{k}$ has the structure of an algebraic variety;
\item For all $k \in \NN$ the inclusion  $\VVV_{k}\into \VVV_{k+1}$ is a closed immersion of algebraic varieties.
\ee
\end{defn}

A {\it morphism} between ind-varieties $\VVV$ and $\WWW$  is a map
$\phi\colon \VVV \to \WWW$  such that for any $k$ there is an $m$
such that $\phi(\VVV_{k}) \subset \WWW_{m}$ and that the induced map
$\VVV_{k}\to \WWW_{m}$ is a morphism of varieties. {\it
Isomorphisms} of ind-varieties are defined in the obvious way.

Two filtrations $\VVV = \bigcup_{k \in \NN} \VVV_{k}$ and $\VVV =
\bigcup_{k \in \NN} \VVV_{k}'$ are called {\it equivalent\/} if for
any $k$ there is an $m$ such that $\VVV_{k}\subset \VVV_{m}'$ is a
closed subvariety as well as $\VVV_{k}'\subset \VVV_{m}$.
Equivalently,  the identity map
$$\id \colon \VVV = \bigcup_{k \in
\NN} \VVV_{k} \to \VVV = \bigcup_{k \in \NN} \VVV_{k}'$$
is an isomorphism of ind-varieties.

\begin{defn}
The {\it Zariski topology} of an ind-variety
$\VVV=\bigcup_{k}\VVV_{k}$ is defined by declaring a subset $U
\subset \VVV$ to be open if the intersections $U \cap\VVV_{k}$ are
Zariski-open in $\VVV_{k}$ for all $k$. It is obvious that $A
\subset \VVV$ is closed if and only if  $A \cap\VVV_{k}$ is
Zariski-closed in $\VVV_{k}$ for all $k$. It follows that a locally
closed subset $\WWW \subset \VVV$ has a natural structure of an
ind-variety, given by the filtration $\WWW_{k}:=\WWW \cap \VVV_{k}$
which are locally closed subvarieties of $\VVV_{k}$. These subsets
are called {\it ind-subvarieties}.

A morphism $\phi\colon \VVV \to \WWW$ is called an {\it immersion}
if the image $\phi(\VVV) \subset \WWW$ is locally closed and $\phi$
induces an isomorphism $\VVV \simto \phi(\VVV)$ of ind-varieties. An
immersion $\phi$ is called a {\it closed (resp. open) immersion} if
$\phi(\VVV) \subset \WWW$ is closed (resp. open).
\end{defn}

\begin{defn}\label{affine-algebraic.def}\strut
\be
\item An ind-variety $\VVV$ is called {\it affine} if it admits a filtration such that all $\VVV_{k}$ are affine. It follows that any filtration of $\VVV$ has this property.
\item The {\it algebra of regular functions\/} on $\VVV=\bigcup \VVV_{k}$ is defined as
$$
\OOO(\VVV):= \Mor(\VVV,\Aone) = \varprojlim \OOO(\VVV_{k})
$$
It will always be regarded as a topological algebra with the obvious
topology as an inverse  limit of finitely generated algebras. The homomorphism
$\phi^{*}\colon\OOO(\WWW) \to \OOO(\VVV)$ induced by any morphism $\phi\colon \VVV \to \WWW$
is continuous. Moreover,
an affine ind-variety $\VVV$ is uniquely determined by the
topological algebra $\OOO(\VVV)$.
\item
The {\it Zariski tangent space\/} of an ind-variety
$\VVV=\bigcup_{k}\VVV_{k}$ is defined in the obvious way:
$$
T_{v}\VVV:=\varinjlim T_{v}\VVV_{k}.
$$
If $\VVV$ is affine, a tangent vector $A \in T_{v}\VVV$ is the
same as a continuous derivation $A\colon\OOO(\VVV)\to \bk$ in $v$.
It is clear that a morphism $\phi\colon \VVV \to \WWW$ between two
ind-varieties induces a linear map between tangent spaces $d\phi_{v}\colon T_{v}\VVV \to
T_{\phi(v)}\WWW$, {\it the differential of $\phi$ in $v$}.
\item
The {\it product of two ind-varieties $\VVV=\bigcup_{k}\VVV_{k}$ and
$\WWW=\bigcup_{j} \WWW_{j}$\/} is the ind-variety defined as
$$\VVV\times\WWW:=\bigcup_{k}\VVV_{k}\times \WWW_{k}.$$
It has the
usual universal properties.
\item An ind-variety $\VVV$ is {\it curve-connected} if for every pair $v,w \in \VVV$ there is an irreducible algebraic curve $C$ and a morphism $\gamma\colon C \to \VVV$ such that $v,w \in\gamma(C)$. One can show that this is equivalent to the existence of a filtration $\VVV=\bigcup_{k}\VVV_{k}$ such that all $\VVV_{k}$ are irreducible (see \cite{FuKr2015On-the-geometry-of}).
\ee
\end{defn}

Since products exist in the category of ind-varieties we can define
ind-groups and ind-semigroups.

\begin{defn}
An {\it ind-group} $\GGG$ is an ind-variety with a group structure
such that multiplication $\GGG\times\GGG \to \GGG$ and inverse $\GGG
\to \GGG$ are morphisms. An {\it ind-semigroup $\SSS$} is defined in
a similar way.
\end{defn}

An {\it action of an ind-group $\GGG$ on a variety $X$} is a
homomorphism $\GGG \to \Aut(X)$ such that the induced map
$\GGG\times X \to X$ is a morphism of ind-varieties. If $X$ is an
affine variety, it is shown in \cite{FuKr2015On-the-geometry-of}
that $\End(X)$ is an affine ind-semigroup and $\Aut(X)$ is an affine
ind-group which is locally closed in $\End(X)$. It follows that an
action of an ind-group $\GGG$ on $X$ is the same as a homomorphism
of ind-groups $\GGG \to \Aut(X)$.

All this carries over to actions of ind-semigroups $\SSS$.

\ps
\subsection{Vector fields and Lie algebras}
A {\it vector field $\delta$} on an affine variety $X$ is a
collection $\delta=(\delta(x))_{x\in X}$ of tangent vectors
$\delta(x)\in T_{x}X$ such that, for all $f \in \OOO(X)$, we have
$\delta f \in \OOO(X)$ where $(\delta f)(x):=\delta(x)f$. It follows
that the vector fields $\VF(X)$ can be identified with the
derivations of $\OOO(X)$ which we denote by $\Der(\OOO(X))$.

The same definition can be used for an affine ind-variety $\VVV$,
and one gets an identification of $\VF(\VVV)$ with the {\it
continuous\/} derivations $\Derc(\OOO(\VVV))$. For an affine
ind-group $\GGG$ one shows that the tangent space $T_{e}\GGG$ has a
natural structure of a Lie algebra. It will be denoted by
$\Lie\GGG$.

If $\GGG$ acts on the variety $X$ and $x \in X$ we denote by
$\mu_{x}\colon \GGG \to X$ the orbit map $g\mapsto gx$.

\begin{prop}\label{vector-fields.prop}
Assume that an affine ind-group $\GGG$ acts on an affine variety
$X$. For  $A \in \Lie \GGG$ and $x\in X$ define the tangent vector
$\xi_{A}(x) \in T_{x}X$ to be the image of $A$ under $d\mu_{x}\colon
\Lie\GGG \to T_{x}X$. Then $\xi_{A}$ is a vector field on $X$. The
resulting linear map $\Xi\colon \Lie\GGG \to \VF(X)$, $A \mapsto
\xi_{A}$, is a anti-homomorphism of Lie algebras.
\end{prop}
\begin{proof}[Outline of Proof]
The action $\phi\colon\GGG\times X \to X$ defines a homomorphism
$\phi^{*}\colon \OOO(X) \to \OOO(\GGG)\otimes\OOO(X)$. Now consider
the following derivation of $\OOO(X)$:
$$
\begin{CD}
\delta\colon \OOO(X) @>{\phi^{*}}>> \OOO(\GGG)\otimes\OOO(X)
@>{A\otimes\id}>> \OOO(X).
\end{CD}
$$
An easy calculation shows that $(\delta f)(x) = A \mu_{x}^{*}(f)  =
d\mu_{x}(A) f$, hence $\delta = \xi_{A}$.
\end{proof}

It is easy to see that this generalizes to the action of an affine
ind-semigroup $\E$ on an affine variety $X$, $\mu\colon\E \to
\End(X)$, and defines a linear map $\Xi\colon T_{\id}\E \to \VF(X)$
whose image $\DDD_{\E}$ are the {\it corresponding vector fields}.


\end{document}